\documentclass[11pt]{article}
\usepackage{a4}
\usepackage[american]{babel}
\usepackage{hyperref}
\usepackage{amsmath, amssymb, amsthm}
\usepackage{stmaryrd}
\usepackage{dsfont}

\newtheorem{theorem}{Theorem}
\newtheorem{proposition}{Proposition}[section]
\newtheorem{lemma}[proposition]{Lemma}
\newtheorem{corollary}[proposition]{Corollary}

\theoremstyle{definition}

\newtheorem*{definition}{Definition}

\newtheorem*{remark}{Remark}

\begin{document}

\title{The local geometry of compact homogeneous Lorentz spaces}

\author{Felix G\"unther\footnote{Institut f\"ur Mathematik, Humboldt-Universit\"at zu Berlin, Unter den Linden 6, 10099 Berlin, Germany. Present address: Isaac Newton Institute for Mathematical Sciences, 20 Clarkson Road, Cambridge CB3 0EH, United Kingdom. E-mail: fguenth@math.tu-berlin.de}}

\date{}
\maketitle

\begin{abstract}
\noindent
In 1995, S.~Adams and G.~Stuck as well as A.~Zeghib independently provided a classification of non-compact Lie groups which can act isometrically and locally effectively on compact Lorentzian manifolds. In the case that the corresponding Lie algebra contains a direct summand isomorphic to the two-dimensional special linear algebra or to a twisted Heisenberg algebra, Zeghib also described the geometric structure of the manifolds. Using these results, we investigate the local geometry of compact homogeneous Lorentz spaces whose isometry groups have non-compact connected components. It turns out that they all are reductive. We investigate the isotropy representation and curvatures. In particular, we obtain that any Ricci-flat compact homogeneous Lorentz space is flat or has compact isometry group.\\ \vspace{0.5ex}

\vspace{0.5ex}

\end{abstract}

\raggedbottom
\setlength{\parindent}{0pt}
\setlength{\parskip}{1ex}

%%%%%%%%%%%%%%%%%%%%%%%%%%%%%%%%%%%%%%%%%%%%%%%%%%%%%%%%%%%%%%%%%%%%%%%%%%%%%%%%%%%%%%%%

\section{Introduction}

%%%%%%%%%%%%%%

The aim of our work is a detailed investigation of compact homogeneous Lorentzian manifolds whose isometry groups have non-compact connected components. We will provide the most important results and proofs of \cite{G11}. Focusing on the concepts and the main ideas, we refer the reader to \cite{G11} for more details.

In Section~\ref{sec:basics}, we introduce the notations and theorems concerning isometric and (locally) effective actions of Lie groups on compact Lorentzian manifolds. Let us shortly describe the content of the basic theorems, Theorems~\ref{th:algebraic} to~\ref{th:geometric_characterization}. Adams and Stuck (\cite{AdSt97a}) as well as Zeghib (\cite{Ze98a}) independently provided an algebraic classification of Lie groups that act isometrically and locally effectively on Lorentzian manifolds that are compact. More generally, Zeghib showed the same result for manifolds of finite volume. We will shortly describe the approach of Zeghib for the proof of Theorem~\ref{th:algebraic_classification}. Theorem~\ref{th:algebraic_classification} states that if a Lie group $G$ is acting isometrically and locally effectively on a Lorentzian manifold $M=(M,g)$ of finite volume, then there exist Lie algebras $\mathfrak{k}$, $\mathfrak{a}$, and $\mathfrak{s}$ such that the Lie algebra $\mathfrak{g}$ of $G$ is equal to the direct sum $\mathfrak{k}\oplus\mathfrak{a}\oplus\mathfrak{s}$. Here, $\mathfrak{k}$ is compact semisimple, $\mathfrak{a}$ is abelian, and $\mathfrak{s}$ is isomorphic to one of the following:
\begin{itemize}
\item the trivial algebra,
\item the two-dimensional affine algebra $\mathfrak{aff}(\mathds{R})$,
\item the $(2d+1)$-dimensional Heisenberg algebra $\mathfrak{he}_d$,
\item a certain $(2d+2)$-dimensional twisted Heisenberg algebra $\mathfrak{he}_d^\lambda$ ($\lambda \in \mathds{Z}_+^d$),
\item the two-dimensional special linear algebra $\mathfrak{sl}_2(\mathds{R})$.
\end{itemize}

The key idea behind the proof is introducing a certain symmetric bilinear form $\kappa$ on $\mathfrak{isom}(M)$, the Lie algebra of  the isometry group of $(M,g)$. For $X,Y \in \mathfrak{isom}(M)$, \[\kappa(X,Y):=\int\limits_M{g(\widetilde{X},\widetilde{Y})(x) d\mu(x)}.\] Here, $\widetilde{X},\widetilde{Y}$ are complete Killing vector fields corresponding to $X,Y$. In the case that $M$ is not compact but has finite volume, one has to restrict the integration to an $\text{Isom}(M)$-invariant non-empty open subset of $M$ such that $|g(\widetilde{X},\widetilde{Y})|$ is bounded by a constant depending only on $X$ and $Y$. Such an open set always exists. $\kappa$ induces in a canonical way a symmetric bilinear form on the Lie algebra $\mathfrak{g}$ of a Lie group $G$ acting isometrically and locally effectively on $M$.

$\kappa$ is ad-invariant and fulfills the so-called condition~\hyperlink{star}{$(\star)$}. A symmetric bilinear form $b$ on the Lie algebra $\mathfrak{g}$ of a Lie group $G$ fulfills this condition if for any subspace $V$ of $\mathfrak{g}$ such that the set of $X \in V$ generating a non-precompact one-parameter group in $G$ is dense in $V$, the restriction of $b$ to $V \times V$ is positive semidefinite, and its kernel has dimension at most one. Proposition~\ref{prop:condition_star} shows that $\kappa$ fulfills this condition.

Condition~\hyperlink{star}{$(\star)$} is the main tool for proving Theorem~\ref{th:algebraic}. Theorem~\ref{th:algebraic} describes the algebraic structure of connected non-compact Lie groups $G$ whose Lie algebras $\mathfrak{g}$ possess an ad-invariant symmetric bilinear form $\kappa$ fulfilling condition~\hyperlink{star}{$(\star)$}. The structure of the Lie algebras is exactly the same as the one in Theorem~\ref{th:algebraic_classification}. Furthermore, in the latter two cases, if $G$ is contained in the isometry group of the manifold, the subgroup generated by $\mathfrak{s}$ has compact center if and only if the subgroup is closed in the isometry group. This formulation corrects the claim of \cite{Ze98a} that the subgroup generated by $\mathfrak{s}$ has compact center.

Theorem~\ref{th:locally_free}, which was stated and shown in its general form in \cite{AdSt97a}, says that the subgroup generated by $\mathfrak{s}$ in Theorem~\ref{th:algebraic_classification} acts locally freely on $M$. This result is important for the geometric characterization of compact Lorentzian manifolds in Theorem~\ref{th:geometric_characterization} (which is due to Zeghib). This theorem considers the case that the Lorentzian manifold $M$ is compact and $\mathfrak{s}$ is isomorphic to the two-dimensional special linear algebra or to a twisted Heisenberg algebra. If $\mathfrak{s}\cong\mathfrak{sl}_2(\mathds{R})$, $M$ is covered isometrically by a warped product of the universal cover of the two-dimensional special linear group and a Riemannian manifold $N$. Else, if $\mathfrak{s}\cong\mathfrak{he}_d^\lambda$, $M$ is covered isometrically by a twisted product $S\times_{Z(S)}N$ of a twisted Heisenberg group $S$ and a Riemannian manifold $N$. Due to the correction of Theorem~\ref{th:algebraic} compared to \cite{Ze98a} and due to the fact that the invariance of a Lorentzian scalar product on a twisted Heisenberg algebra under the adjoint action of the nilradical (what Zeghib called \textit{essential ad-invariance}) is equivalent to ad-invariance (see Proposition~\ref{prop:lorentz_heisenberg}~(iii)), the formulation of Theorem~\ref{th:geometric_characterization} is different from the one in \cite{Ze98a}.

Our main theorems are formulated in Section~\ref{sec:main_theorems}. We will show them in Section~\ref{sec:homogeneous}, where we provide an analysis of compact homogeneous Lorentz spaces $M$, in particular of those whose isometry groups have non-compact connected components. We start by presenting a topological and geometric description of them in Theorem~\ref{th:homogeneous_characterization}, which is also slightly different from the corresponding result of Zeghib. Essentially, it shows that if the isometry group of $M$ has non-compact connected components, $M$ is covered isometrically either by a metric product of the universal cover of the two-dimensional special linear group and a compact homogeneous Riemannian manifold $N$, or by a twisted product $S\times_{Z(S)}N$ of a twisted Heisenberg group $S$ and a compact homogeneous Riemannian manifold $N$. Additionally, we give a covering of the identity component of $\textnormal{Isom}(M)$, which will be important for the investigation of the local geometry of $M$.

Theorem~\ref{th:homogeneous_reductive} shows what was originally stated in \cite{Ze98a}, namely that the connected components of the isotropy group are compact. We give an elegant proof using the ideas of Adams and Stuck in the proof of Theorem~\ref{th:locally_free}. Moreover, it turns out that every compact homogeneous Lorentzian manifold has a reductive representation defined in a natural way. For this representation, the induced bilinear form $\kappa$ plays an essential role.

Using a slightly different reductive representation than in Theorem~\ref{th:homogeneous_reductive} in the case that the isometry group contains a twisted Heisenberg group as a subgroup, we are able to describe the local geometry of compact homogeneous Lorentzian manifolds whose isometry groups have non-compact connected components in terms of the curvature of the manifold. We also investigate the isotropy representation of the manifold, our result concerning a decomposition into (weakly) irreducible summands being summarized in Theorem~\ref{th:isotropy_representation}.

Our results of Paragraph~\ref{par:curvature} directly yield the proof of Theorem~\ref{th:homogeneous_not_Ricci_flat} which states that the isometry group of any Ricci-flat compact homogeneous Lorentzian manifold has compact connected components. Together with two results in \cite{PiZe10} and \cite{RoSa96}, it follows that the isometry group of any Ricci-flat compact homogeneous Lorentzian manifold that is non-flat is in fact compact. Note that certain Cahen-Wallach spaces, which all are symmetric Lorentzian spaces, are Ricci-flat and non-flat. It is not known whether at least all Ricci-flat compact homogeneous Lorentzian manifolds are flat.

%%%%%%%%%%%%%%%%%%%%%%%%%%%%%%%%%%%%%%%%%%%%%%%%%%%%%%%%%%%%%%%%%%%%%%%%%%%%%%%%%%%%%%%%

\section{Lie groups acting isometrically on compact Lorentzian manifolds}\label{sec:basics}

This section is devoted to describe the classification of Lie algebras of Lie groups acting isometrically and locally effectively on compact Lorentzian manifolds. Following the approach of Zeghib in \cite{Ze98a}, a certain symmetric bilinear form $\kappa$ on Lie algebras plays an important role. We also state a theorem about local freeness of the action of a subgroup of the isometry group, which is due to \cite{AdSt97a}. Following \cite{Ze98a}, we finally describe the topological and geometric structure of the manifolds.

%%%%%%%%%%%%%%%%%%%%%%%%%%%%%%%%%%%%%%%%

\subsection{Notations and induced bilinear form}\label{par:notations}

Let $M=(M,g)$ be a semi-Riemannian manifold and $\mu$ the induced Lebesgue measure on $M$. We will consider only connected smooth real manifolds without boundary. Moreover, all Lie groups and Lie algebras will be real and finite-dimensional. Unless otherwise stated, all actions of Lie groups are smooth. If a Lie group $G$ acts isometrically on $M$, we identify a group element $f$ with the corresponding isometry, which allows us to speak about the differential $df$.

The following proposition is a classical result giving an identification of $\mathfrak{isom}(M)$, the Lie algebra of the isometry group $\textnormal{Isom}(M)$, with $\mathfrak{kill}_c(M)$, the Lie algebra of complete Killing vector fields, as vector spaces. One can find a proof in \cite[Proposition~9.33]{ON83}.

\begin{proposition}\label{prop:isometry_Killing}
For a semi-Riemannian manifold $M$, the map $\mathfrak{isom}(M) \to \mathfrak{kill}_c(M)$ defined by ${X \mapsto \widetilde{X}}$, $\widetilde{X}(x):=\frac{\partial}{\partial t}(\exp(tX) \cdot x) |_{t=0}$, is an anti-isomorphism of Lie algebras, that is, $[\widetilde{X},\widetilde{Y}]=-\widetilde{[X,Y]}$.
\end{proposition}

\begin{definition}
Let $M=(M,g)$ be a compact semi-Riemannian manifold and $G$ a Lie group acting isometrically and locally effectively on $M$. Then $\kappa$ is the \textit{induced bilinear form} on its Lie algebra $\mathfrak{g}$ defined by
\begin{equation*}
\kappa(X,Y):=\int\limits_M{g(\widetilde{X},\widetilde{Y})(x) d\mu(x)}.
\end{equation*}
\end{definition}

It can be easily checked that $\kappa$ is Ad-invariant. Moreover, in the case that the metric $g$ is Lorentzian, $\kappa$ fulfills the following remarkable property.

\begin{proposition}\label{prop:condition_star}
Let $M=(M,g)$ be a compact Lorentzian manifold and $V$ a subspace of $\mathfrak{isom}(M)$ such that the set of $X \in V$ generating a non-precompact one-parameter group of isometries is dense in $V$. Then the restriction of the induced bilinear form $\kappa$ to $V \times V$ is positive semidefinite and the dimension of its kernel is at most one.
\end{proposition}
\begin{proof}
See \cite[Corollary~1.21]{G11}. This proof is due to \cite{Ze98a}. We essentially show that any Killing vector field $\widetilde{X}$ which is timelike somewhere generates a precompact one-parameter group. For this, we change the Lorentzian metric in a small neighborhood where $\widetilde{X}$ is timelike into a Riemannian metric and look at the embedding of the isometry group into the bundle of orthonormal frames.
\end{proof}

\begin{remark}
In the case that $M$ is not compact but has finite volume, $\kappa$ is well-defined if we restrict the integral to a certain $G$-invariant open subset $U \subset M$. Proposition~\ref{prop:condition_star} is also satisfied. The proof of the existence of such a subset $U$ is due to Zeghib and can be found in \cite[Proposition~1.16]{G11}. Roughly speaking, we are looking at the Gau{\ss} map, which maps each point of $M$ to the symmetric bilinear form on $\mathfrak{g}$ induced by $g$ and apply the F{\"u}rstenberg lemma, which essentially states that a connected Lie group acting on a vector space and preserving a finite measure acts precompactly on the support of the measure (see \cite[Corollary~1.18]{G11}).

The example before \cite[Proposition~1.16]{G11} shows that the restriction of the integral to some subset is necessary in general.
\end{remark}
%%%%%%%%%%%%%%%%%%%%%%%%%%%%%%%%%%%%%%%%

\subsection{Lie algebras appearing in the classification theorem}\label{par:algebras}

We can consider the Lie algebra $\mathfrak{sl}_2(\mathds{R})$ of the two-dimensional special linear group $\textnormal{SL}_2(\mathds{R})$ as the subalgebra of the general linear algebra $\mathfrak{gl}_2(\mathds{R})$ that is spanned by \[e:= \begin{pmatrix} 0 & 1 \\ 0 & 0\end{pmatrix}, \ f:=\begin{pmatrix} 0 & 0 \\ 1 & 0\end{pmatrix}, \text{ and } h:=\begin{pmatrix} 1 & 0 \\ 0 & -1\end{pmatrix}.\] They satisfy the relations $[e,f]=h$, $[h,e]=2e$, and $[h,f]=-2f$ and therefore form an $\mathfrak{sl}_2$-triple.

\begin{definition}
A triple $\left\{e,f,h\right\}$ of elements in a Lie algebra $\mathfrak{g}$ satisfying $[e,f]=h$, $[h,e]=2e$, and $[h,f]=-2f$, is called an \textit{$\mathfrak{sl}_2$-triple}.
\end{definition}

We can identify the Lie algebra $\mathfrak{aff}(\mathds{R})$ with the subalgebra of $\mathfrak{sl}_2(\mathds{R})$ spanned by \[X:= \begin{pmatrix} \frac{1}{2} & 0 \\ 0 & -\frac{1}{2}\end{pmatrix} \text{ and } Y:= \begin{pmatrix} 0 & 1 \\ 0 & 0\end{pmatrix}.\] They satisfy the relation $[X,Y]=Y$.

Note that the Killing form of $\mathfrak{sl}_2(\mathds{R})$ is Lorentzian, whereas the Killing form of $\mathfrak{aff}(\mathds{R})$ is positive semidefinite and the kernel is exactly the span of $Y$.

A Lie algebra $\mathfrak{g}$ is a $2d+1$-dimensional Heisenberg algebra $\mathfrak{he}_d$ if and only if its center $\mathfrak{z(g)}$ is one-dimensional (spanned by $Z$) and there is a complementary vector space $V \subset \mathfrak{g}$ and a non-degenerate alternating bilinear form $\omega$ such that \[[X,Y]=\omega(X,Y)Z\] for all $X,Y \in V$. 

\begin{definition}
A set $\left\{Z, X_1, Y_1, \ldots, X_d, Y_d\right\}$ of elements in a Heisenberg algebra such that $Z$ generates the center and $\left\{X_1, Y_1, \ldots, X_d, Y_d\right\}$ is a symplectic basis of the complementary space $V$, that is, \[\omega(X_k,Y_j)=\delta_{jk} \text{ and } \omega(X_k,X_j)=0=\omega(Y_k,Y_j)\] for all $j$ and $k$, is called a \textit{canonical basis}. A \textit{Heisenberg group} is a Lie group with Lie algebra $\mathfrak{he}_d$.
\end{definition}

\begin{definition} $\widetilde{\textnormal{He}_d}$ denotes the simply-connected Lie group with Lie algebra $\mathfrak{he}_d$. For a (uniform) lattice $\Lambda$ of the center (isomorphic to $\mathds{Z}$), the quotient $\widetilde{\textnormal{He}_d}/\Lambda$ is denoted by $\textnormal{He}_d$.
\end{definition}

\begin{remark}
The quotient $\textnormal{He}_d$ is up to isomorphism independent of the choice of $\Lambda$.
\end{remark}

Let $\lambda=(\lambda_1,\ldots,\lambda_d) \in \mathds{R}_+^d$, $d> 0$. The corresponding twisted Heisenberg algebra $\mathfrak{he}_d^\lambda$ of dimension $2d+2$ is spanned by the elements $T,Z, X_1, Y_1, \ldots, X_d, Y_d$ and the non-vanishing Lie brackets are given by \[[X_k,Y_k]=\lambda_k Z, \ [T,X_k]=\lambda_k Y_k, \text{ and } [T,Y_k]=-\lambda_k X_k,\] $k=1,\ldots,d$. Thus, $\mathfrak{he}_d^\lambda=\mathds{R} T \inplus \mathfrak{he}_d$ is a semidirect sum, where $\mathfrak{he}_d$ can be identified with the subalgebra spanned by $Z, X_1, Y_1, \ldots, X_d, Y_d$.

\begin{definition}
A set $\left\{T, Z, X_1, Y_1, \ldots, X_d, Y_d\right\}$ of elements in a twisted Heisenberg algebra fulfilling the same relations as above is called a \textit{canonical basis}. A \textit{twisted Heisenberg group} is a Lie group with Lie algebra $\mathfrak{he}_d^\lambda$.
\end{definition}

Let $\widetilde{\textnormal{He}_d^\lambda}$ be the simply-connected Lie group with Lie algebra $\mathfrak{he}_d^\lambda$ and $\exp: \mathfrak{he}_d^\lambda \to \widetilde{\textnormal{He}_d^\lambda}$ be the exponential. Then, $\widetilde{\textnormal{He}_d^\lambda}=\exp(\mathds{R} T) \ltimes \widetilde{\textnormal{He}_d}$, where $\widetilde{\textnormal{He}_d}$ can be identified with $\exp(\mathfrak{he}_d)$.

If $\lambda \in \mathds{Q}^d_+$, then $\Lambda^\prime:=\textnormal{ker}(\exp \circ \textnormal{ad}: \mathds{R}T \to \textnormal{Aut}(\mathfrak{he}_d^\lambda))$ is a uniform lattice in $\mathds{R}T$. As in the case of the Heisenberg algebra, let $\Lambda$ be a lattice of the center of $\widetilde{\textnormal{He}_d}$.

\begin{definition}
Let $\lambda \in \mathds{Q}^d_+$. Then, $\textnormal{He}_d^\lambda:=\widetilde{\textnormal{He}_d^\lambda}/(\Lambda^\prime \times \Lambda)$. In an analogous way, we define $\overline{\textnormal{He}_d^\lambda}:=\widetilde{\textnormal{He}_d^\lambda}/\Lambda^\prime$.
\end{definition}
\begin{remark}
We have $\textnormal{He}_d^\lambda\cong \mathds{S}^1 \ltimes \textnormal{He}_d$ and $\overline{\textnormal{He}_d^\lambda} \cong \mathds{S}^1 \ltimes \widetilde{\textnormal{He}_d}$.
\end{remark}

By some algebraic calculations, we show that we can restrict to $\lambda \in \mathds{Z}^d_+$, see \cite[Lemma~3.2]{AdSt97a} or \cite[Lemma~1.11]{G11}.

\begin{lemma}\label{lem:isom_heisenberg}
The isomorphism classes of $\widetilde{\textnormal{He}_d^\lambda}$, $\lambda \in \mathds{Q}^d_+$, are in one-to-one-correspondence with the set $\mathds{Z}_+^d/\sim$, where $\lambda \sim \eta$ if and only if there exists $a \in \mathds{\mathds{R}}_+$, such that \[ \left\{\lambda_1,\ldots,\lambda_d\right\}=\left\{a\eta_1,\ldots,a\eta_d\right\}.\]
\end{lemma}

On any twisted Heisenberg algebra, there exists up to isomorphism only one ad-invariant Lorentz form. Astonishingly, invariance under $\textnormal{ad}(\mathfrak{he}_d)$ is equivalent to invariance under $\textnormal{ad}(\mathfrak{he}_d^\lambda)$ in this case. That is the main reason for our difference later in Theorem~\ref{th:geometric_characterization} to the original theorem in \cite{Ze98a}.

\begin{proposition}\label{prop:lorentz_heisenberg}
The following are true:

\begin{enumerate}
\item A twisted Heisenberg algebra $\mathfrak{he}_d^\lambda$ admits an ad-invariant Lorentz form.
\item Any two ad-invariant Lorentzian scalar products $b_1,b_2$ on $\mathfrak{he}_d^\lambda$ are equivalent, that is, there is an automorphism $L:\mathfrak{he}_d^\lambda \to \mathfrak{he}_d^\lambda$ such that $b_1(X,Y)=b_2(L(X),L(Y))$ for all $X,Y \in \mathfrak{he}_d^\lambda$.
\item A Lorentzian scalar product on $\mathfrak{he}_d^\lambda$ is $\textnormal{ad}(\mathfrak{he}_d)$-invariant if and only if it is $\textnormal{ad}(\mathfrak{he}_d^\lambda)$-invariant.
\end{enumerate}
\end{proposition}
\begin{proof}
(i) Let $\left\{T, Z, X_1, Y_1, \ldots, X_d, Y_d\right\}$ be a canonical basis of $\mathfrak{he}_d^\lambda$. Now define a Lorentzian scalar product $\langle \cdot,\cdot \rangle$ such that $\left\{ X_1, Y_1, \ldots, X_d, Y_d\right\}$ is orthonormal and orthogonal to $\left\{T,Z\right\}$, ${\langle T,T \rangle=\langle Z,Z \rangle=0}$, and $\langle T,Z \rangle=1$. It is easy to check that this scalar product is ad-invariant, see \cite[Equation~(1)]{AdSt97a} or \cite[Proposition~1.13~(i)]{G11}.

(ii) See \cite[Proposition~5.2]{Ze98a} or \cite[Proposition~4.21]{G11}. The idea is to show that any ad-invariant Lorentzian scalar product on $\mathfrak{he}_d^\lambda$ is determined by two real parameters and then to write down the automorphisms explicitly.

(iii) This is done similarly to the first two parts, see \cite[Proposition~1.13~(ii)]{G11}.
\end{proof}

%%%%%%%%%%%%%%%%%%%%%%%%%%%%%%%%%%%%%%%%

\subsection{Basic theorems}\label{par:basic_theorems}

Following \cite{Ze98a}, the key theorem providing the classification of the Lie algebras is the following algebraic result.

\begin{theorem}\label{th:algebraic}
Let $G$ be a connected non-compact Lie group. Assume that there is an ad-invariant symmetric bilinear form $\kappa$ on its Lie algebra $\mathfrak{g}$ that fulfills the following non-degeneracy condition:

\par
\begingroup
\leftskip=1em
\noindent \hypertarget{star}{$(\star)$} Let $V$ be a subspace of $\mathfrak{g}$ such that the set of $X \in V$ generating a non-precompact one-parameter group in $G$ is dense in $V$. Then, the restriction of $\kappa$ to $V \times V$ is positive semidefinite and its kernel has dimension at most one.
\par
\endgroup

Then, $\mathfrak{g}=\mathfrak{k}\oplus\mathfrak{a}\oplus\mathfrak{s}$ is a $\kappa$-orthogonal direct sum of a compact semisimple Lie algebra $\mathfrak{k}$, an abelian algebra $\mathfrak{a}$, and a Lie algebra $\mathfrak{s}$ that is either trivial, isomorphic to $\mathfrak{aff}(\mathds{R})$, to a Heisenberg algebra $\mathfrak{he}_d$, to a twisted Heisenberg algebra $\mathfrak{he}_d^\lambda$ with $\lambda \in \mathds{Z}_+^d$, or to $\mathfrak{sl}_2(\mathds{R})$.

Furthermore, the following are true:
\begin{enumerate}
\item If $\mathfrak{s}$ is trivial, $\kappa$ is positive semidefinite and its kernel has dimension at most one. If $\mathfrak{s}$ is non-trivial, $\kappa$ restricted to $\mathfrak{a} \times \mathfrak{a}$ and $\mathfrak{k} \times \mathfrak{k}$ is positive definite.

\item Let $\mathfrak{s}\cong\mathfrak{aff}(\mathds{R})$. Then, the restriction of $\kappa$ to $\mathfrak{s} \times \mathfrak{s}$ is positive semidefinite and its kernel is exactly the span of the generator of the translations in the affine group.

\item Let $\mathfrak{s}\cong\mathfrak{he}_d$. Then, the restriction of $\kappa$ to $\mathfrak{s} \times \mathfrak{s}$ is positive semidefinite and its kernel is exactly the center of $\mathfrak{he}_d$.

\item Let $\mathfrak{s}\cong\mathfrak{he}_d^\lambda$. Then, $\kappa$ restricted to $\mathfrak{s} \times \mathfrak{s}$ is a Lorentz form. The subgroup in $G$ generated by $\mathfrak{s}$ is isomorphic to $\textnormal{He}_d^\lambda$ or $\overline{\textnormal{He}_d^\lambda}$ if it is closed in $G$ or not, respectively. Moreover, the abelian subgroup generated by $\mathfrak{a}\oplus\mathfrak{z(s)}$ is compact.

\item Let $\mathfrak{s}\cong\mathfrak{sl}_2(\mathds{R})$. Then, $\kappa$ restricted to $\mathfrak{s} \times \mathfrak{s}$ is a positive multiple of the Killing form of $\mathfrak{s}$. The subgroup in $G$ generated by $\mathfrak{s}$ is isomorphic to some $\textnormal{PSL}_k(2,\mathds{R})$, the $k$-covering of $\textnormal{PSL}(2,\mathds{R})$, if and only if it is closed in $G$. Moreover, the abelian subgroup generated by $\mathfrak{a}$ is compact.
\end{enumerate}
\end{theorem}

\begin{remark}
The proof consists essentially of case distinctions and algebraic calculations. One can find the proof in \cite[Chapter~3 (Theorem~1)]{G11}, where in the beginning a diagram shows a summary of the proof. The proof is mainly due to \cite[Th{\'e}or{\`e}me~alg{\'e}brique~1.11]{Ze98a}, where the following was missing: For both cases $\mathfrak{s}\cong\mathfrak{sl}_2(\mathds{R})$ and $\mathfrak{s}\cong\mathfrak{he}_d^\lambda$, $\lambda \in \mathds{Z}_+^d$, 
there exist a connected non-compact Lie group $G$ and a symmetric bilinear form $\kappa$ on its Lie algebra $\mathfrak{g}$ fulfilling condition~\hyperlink{star}{$(\star)$} such that $\mathfrak{g}=\mathfrak{s}\oplus\mathds{R}$ and the subgroup generated by $\mathfrak{s}$ is not closed in $G$. In particular, it is not isomorphic to a finite covering of $\textnormal{PSL}(2,\mathds{R})$ or isomorphic to $\textnormal{He}_d^\lambda$, respectively. Examples are given in \cite[Propositions~3.19 and~3.24]{G11}.
\end{remark}

According to Proposition~\ref{prop:condition_star}, Theorem~\ref{th:algebraic} applies to a connected non-compact closed Lie subgroup $G$ of $\textnormal{Isom}(M)$, where $M$ is a compact Lorentzian manifold. More generally, we obtain the following classification result \cite[Theorem~2]{G11}.

\begin{theorem}\label{th:algebraic_classification}
Let $M$ be a compact Lorentzian manifold and $G$ a connected Lie group acting isometrically and locally effectively on $M$. Then, its Lie algebra $\mathfrak{g}=\mathfrak{k}\oplus\mathfrak{a}\oplus\mathfrak{s}$ is a direct sum of a compact semisimple Lie algebra $\mathfrak{k}$, an abelian algebra $\mathfrak{a}$, and a Lie algebra $\mathfrak{s}$ that is either trivial, isomorphic to $\mathfrak{aff}(\mathds{R})$, to a Heisenberg algebra $\mathfrak{he}_d$, to a twisted Heisenberg algebra $\mathfrak{he}_d^\lambda$ with $\lambda \in \mathds{Z}_+^d$, or to $\mathfrak{sl}_2(\mathds{R})$.

Suppose that $G$ is a Lie subgroup of the isometry group $\textnormal{Isom}(M)$. If $\mathfrak{s}\cong\mathfrak{he}_d^\lambda$, the subgroup generated by $\mathfrak{s}$ is isomorphic to $\textnormal{He}_d^\lambda$ or $\overline{\textnormal{He}_d^\lambda}$ if it is closed in $\textnormal{Isom}(M)$ or not, respectively. If $\mathfrak{s}\cong\mathfrak{sl}_2(\mathds{R})$, the subgroup generated by $\mathfrak{s}$ is isomorphic to some $\textnormal{PSL}_k(2,\mathds{R})$ if and only if it is closed in $\textnormal{Isom}(M)$.
\end{theorem}

\begin{remark}
The first part of the theorem was shown independently by Adams and Stuck in \cite[Theorem~11.1]{AdSt97a} and Zeghib in \cite{Ze98a}, where the latter proof applies also in the case of manifolds of finite volume. For the second part, it is not known to the author whether there exist compact Lorentzian manifolds (or of finite volume) such that $\overline{\textnormal{He}_d^\lambda}$ or a finite quotient of $\widetilde{\textnormal{SL}_2(\mathds{R})}$ is a subgroup of $\textnormal{Isom}(M)$.
\end{remark}

The following theorem will be important for the geometric characterization of the investigated manifolds.

\begin{theorem}\label{th:locally_free}
Let $M$ be a compact Lorentzian manifold and $G$ a connected Lie group acting isometrically and locally effectively on $M$. Then, the action of the subgroup $S$ generated by the direct summand $\mathfrak{s}$ in its Lie algebra $\mathfrak{g}=\mathfrak{k}\oplus\mathfrak{a}\oplus\mathfrak{s}$ (cf.~Theorem~\ref{th:algebraic_classification}) is locally free.
\end{theorem}

\begin{remark}
One can find the proof in \cite[Theorem~11.1]{AdSt97a} or in \cite[Theorem~4]{G11}. The (algebraic) proof uses the structure of the algebra given in Theorem~\ref{th:algebraic_classification} and is based on two lemmas. The first one states that non-trivial lightlike complete Killing vector fields on (connected) Lorentzian manifolds are nowhere vanishing (see \cite[Lemma~6.1]{AdSt97a} or \cite[Lemma~4.3]{G11}). The other one is the following (see \cite[Lemma~6.5]{AdSt97a} or \cite[Lemma~4.4]{G11}).

\begin{lemma}\label{lem:stabilizer}
Let $G$ be a connected Lie group with Lie algebra $\mathfrak{g}$ acting continuously on a compact topological space $M$.

Suppose there are $x \in M$, $X \in \mathfrak{g}, Y \in \mathfrak{g}_x$, and $k \in \mathds{Z}_+$ such that $[X,Z]=0$ for $Z=\textnormal{ad}_X^k(Y)$. Then, there exists $y \in M$ such that $Z \in \mathfrak{g}_y$. Here, $\mathfrak{g}_x$ and $\mathfrak{g}_y$, respectively, denote the Lie algebras of the subgroups of $G$ fixing $x$ and $y$, respectively.
\end{lemma}

Note that the cases $\mathfrak{s}\cong\mathfrak{he}_d^\lambda$ and $\mathfrak{s}\cong\mathfrak{sl}_2(\mathds{R})$ of Theorem~\ref{th:locally_free} are also done in \cite[Th{\'e}or{\`e}me~g{\'e}om{\'e}trique~1.12]{Ze98a}, but in a different manner. 
\end{remark}

The geometric characterization of compact Lorentzian manifolds with a non-compact connected component of the identity in the isometry group is given by the following theorem.

\begin{theorem}\label{th:geometric_characterization}
Let $M$ be a compact Lorentzian manifold and $G$ a connected closed non-compact Lie subgroup of the isometry group $\textnormal{Isom}(M)$. According to Theorem~\ref{th:algebraic_classification}, its Lie algebra is a direct sum $\mathfrak{g}=\mathfrak{k}\oplus\mathfrak{a}\oplus\mathfrak{s}$ as described in the theorem.
\begin{enumerate}
\item If the induced bilinear form $\kappa$ is positive semidefinite, then $\mathfrak{s}$ is neither isomorphic to $\mathfrak{sl}_2(\mathds{R})$ nor $\mathfrak{he}_d^\lambda$. The orbits of $G$ are nowhere timelike and the kernel of $\kappa$ is either trivial or the span of a lightlike Killing vector field with geodesic orbits.

\item If $\mathfrak{s}\cong\mathfrak{sl}_2(\mathds{R})$, $M$ is covered isometrically by a warped product $N \times_\sigma \widetilde{\textnormal{SL}_2(\mathds{R})}$ of a Riemannian manifold $N=(N,h)$ and the universal cover $\widetilde{\textnormal{SL}_2(\mathds{R})}$ of $\textnormal{SL}_2(\mathds{R})$ furnished with the bi-invariant metric given by the Killing form $k$ of $\mathfrak{sl}_2(\mathds{R})$, that is, $M$ is covered by the manifold $N \times \widetilde{\textnormal{SL}_2(\mathds{R})}$ with the metric $g_{(x,\cdot)}=h_x \times \left(\sigma^2\left(x\right)k\right)$, $\sigma:N \to \mathds{R}_+$ smooth.

Moreover, there is a discrete subgroup $\Gamma$ in $\textnormal{Isom}(N)\times\widetilde{\textnormal{SL}_2(\mathds{R})}$ acting freely on $N \times_\sigma \widetilde{\textnormal{SL}_2(\mathds{R})}$ such that $M\cong\Gamma \backslash\mkern-5mu \left(N\times_\sigma \widetilde{\textnormal{SL}_2(\mathds{R})}\right)$.

\item Let $\mathfrak{s}\cong\mathfrak{he}_d^\lambda$ and $S$ be the subgroup generated by $\mathfrak{s}$. $S$ is isomorphic to $\textnormal{He}_d^\lambda$ or $\overline{\textnormal{He}_d^\lambda}$ and the center $Z(S)$ is isomorphic to $\mathds{S}^1$ or $\mathds{R}$ if $S$ is closed in $\textnormal{Isom}(M)$ or not, respectively.

Consider the space $\mathcal{M}$ of ad-invariant Lorentz forms on $\mathfrak{s}$. Then, there exist a Riemannian manifold $N=(N,h)$ with a locally free and isometric $Z(S)$-action and a smooth map $m: N \to \mathcal{M}$ invariant under the $Z(S)$-action such that $M$ is covered isometrically by the Lorentzian manifold $S \times_{Z(S)} N$, constructed in the following way:

Consider the product $S\times N$ furnished with the metric $g_{(\cdot,x)}=m(x) \times h_x$, where the $S$-factor is provided with the bi-invariant metric defined by $m(x)$. Let $\mathcal{O}$ be the distribution of $N$ orthogonal to the $Z(S)$-orbits and $\widetilde{\mathcal{S}}$ be the distribution of $S$ given by the tangent spaces.

The center $Z(S)$ acts isometrically on $S$ by multiplication of the inverse. The action of a central element $z$ that maps $f$ to $fz^{-1}=z^{-1}f$ corresponds to a translation in the center component. Thus, we have a locally free and isometric action of $Z(S)$ on $S\times N$ by the diagonal action. Factorizing through this action, we obtain the quotient space $S\times_{Z(S)} N$. This is a manifold and we can provide it with the metric given by projection of the induced metric on $\widetilde{\mathcal{S}}\oplus\mathcal{O}$.

Furthermore, there is a discrete subgroup $\Gamma$ in $S\times_{Z(S)}\textnormal{Isom}_{Z(S)}(N)$, such that $M\cong\Gamma \backslash \mkern-5mu  \left(S\times_{Z(S)}N\right)$. Here, the quotient group defined in the same way as the quotient manifold above and $\textnormal{Isom}_{Z(S)}(N)$ is the group of $Z(S)$-equivariant isometries of $N$ acting freely on the manifold $S\times_{Z(S)}N$.
\end{enumerate}
\end{theorem}

\begin{remark}
This theorem is due to Zeghib \cite[Th{\'e}or{\`e}me~g{\'e}om{\'e}trique~1.12 and Th{\'e}or{\`e}me~1.14]{Ze98a} and corresponds to \cite[Theorem~5]{G11}. The main difference is in the third part of the theorem, where the equivalence of $\textnormal{ad}(\mathfrak{he}_d)$-invariance and $\textnormal{ad}(\mathfrak{he}_d^\lambda)$-invariance of Lorentzian scalar products on $\mathfrak{he}_d^\lambda$ as well as the difference in Theorem~\ref{th:algebraic} are taken into account.

The proof of the first part of the theorem uses the structure of the Lie algebra given in Theorem~\ref{th:algebraic} and Proposition~\ref{prop:condition_star}.

For the second two parts, let $S$ be the subgroup generated by $\mathfrak{s}$. We first show algebraically that the orbits $\mathcal{S}$ of the action of $S$ on $M$ are of Lorentzian character everywhere. Denoting by $\mathcal{O}=\mathcal{S}^\perp$ the orthogonal distribution and by $\mathcal{Z}$ the distribution induced by the action of $Z(S)$ in the case $\mathfrak{s}\cong\mathfrak{he}_d^\lambda$, the distributions $\mathcal{O}$ and $\mathcal{O}+\mathcal{Z}$, respectively, are involutive. We take a leaf $N$ and furnish it with a Riemannian metric such that it coincides on $\mathcal{O}$ with the metric of $M$ and $\mathcal{Z}$ is orthogonal to $\mathcal{O}$.

If we furnish $S$ with a right-invariant Riemannian metric, we can use the theory of Riemannian coverings to obtain that $N \times S \to M$ and $S \times_{Z(S)} N \to M$, respectively, are coverings. These maps are canonically defined through the action of $S$ on $M$.

Some calculations show that the deck transformations $\Gamma$ are acting transitively on the fibers and are induced by elements of $\textnormal{Isom}(N) \times S$ or $S \times_{Z(S)} \textnormal{Isom}_{Z(S)}(N)$, respectively. $(\psi_{\gamma},\gamma) \in \textnormal{Isom}(N) \times S$ is acting on $N \times S$ by \[(\psi_\gamma,\gamma)(x,f)\mapsto (\psi_\gamma(x),f\gamma)\] and $[\gamma,\psi_\gamma] \in \Gamma \subset S \times_{Z(S)} \textnormal{Isom}_{Z(S)}(N)$ is acting on $S \times_{Z(S)} N$ by \[[\gamma,\psi_\gamma]([f,x])\mapsto [f\gamma,\psi_\gamma(x)].\]

Pulling back the metric along the orbits $\mathcal{S}$ gives a bi-invariant metric on $S$. In the case $\mathfrak{s}\cong\mathfrak{sl}_2(\mathds{R})$ this follows from algebraic calculations, in the case $\mathfrak{s}\cong\mathfrak{he}_d^\lambda$ the F{\"u}rstenberg lemma plays an important role.
\end{remark}

%%%%%%%%%%%%%%%%%%%%%%%%%%%%%%%%%%%%%%%%%%%%%%%%%%%%%%%%%%%%%%%%%%%%%%%%%%%%%%%%%%%%%%%%

\section{Main theorems}\label{sec:main_theorems}

In this section, we describe the main theorems of this paper.

\begin{theorem}\label{th:homogeneous_characterization}
Let $M$ be a compact homogeneous Lorentzian manifold. Suppose $\textnormal{Isom}^0(M)$, the connected component of the identity in the isometry group, is not compact. Let $\mathfrak{isom}(M)=\mathfrak{k}\oplus\mathfrak{a}\oplus\mathfrak{s}$ be the decomposition of its Lie algebra according to Theorem~\ref{th:algebraic_classification}.

Then either $\mathfrak{s}\cong\mathfrak{sl}_2(\mathds{R})$ or $\mathfrak{s}\cong\mathfrak{he}_d^\lambda$.
\begin{enumerate}
\item If $\mathfrak{s}\cong\mathfrak{sl}_2(\mathds{R})$, $M$ is covered isometrically by the metric product $N \times \widetilde{\textnormal{SL}_2(\mathds{R})}$ of a compact homogeneous Riemannian manifold $N=(N,h)$ and the universal cover $\widetilde{\textnormal{SL}_2(\mathds{R})}$ of $\textnormal{SL}_2(\mathds{R})$ furnished with the metric given by a positive multiple of the Killing form $k$ of $\mathfrak{sl}_2(\mathds{R})$.

Furthermore, there is a uniform lattice $\Gamma_0$ in $\widetilde{\textnormal{SL}_2(\mathds{R})}$ and a group homomorphism ${\varrho:\Gamma_0 \to \textnormal{Isom}(N)}$ such that $M\cong\Gamma \backslash \mkern-5mu \left(N \times \widetilde{\textnormal{SL}_2(\mathds{R})}\right)$, where $\Gamma$ is the graph of $\varrho$. The centralizer of $\Gamma$ in the group ${\textnormal{Isom}(N \times \widetilde{\textnormal{SL}_2(\mathds{R})})}$ acts transitively on $N \times \widetilde{\textnormal{SL}_2(\mathds{R})}$.

$\textnormal{Isom}^0(M)$ is isomorphic to a central quotient of $C \times \widetilde{\textnormal{SL}_2(\mathds{R})}$, where $C$ is the connected component of the identity in the centralizer of $\varrho(\Gamma_0)$ in $\textnormal{Isom}(N)$. $C$ acts transitively on $N$.

\item Let $\mathfrak{s}\cong\mathfrak{he}_d^\lambda$ and $S$ be the subgroup generated by $\mathfrak{s}$. $S$ is isomorphic to $\textnormal{He}_d^\lambda$ or $\overline{\textnormal{He}_d^\lambda}$ and the center $Z(S)$ is isomorphic to $\mathds{S}^1$ or $\mathds{R}$, if $S$ is closed in $\textnormal{Isom}^0(M)$ or not, respectively.

Then, there exist a compact homogeneous Riemannian manifold $N=(N,h)$ with a locally free and isometric $Z(S)$-action and an ad-invariant Lorentz form $m$ on $\mathfrak{s}$ such that $M$ is covered isometrically by the Lorentzian manifold $S\times_{Z(S)} N$, constructed in the same way as in Theorem~\ref{th:geometric_characterization} but with constant $m$.

Moreover, there is a discrete subgroup $\Gamma$ in $S \times_{Z(S)} \textnormal{Isom}_{Z(S)}(N)$, where the latter group is constructed similarly to the space $S\times_{Z(S)} N$ and $\textnormal{Isom}_{Z(S)}(N)$ denotes the group of $Z(S)$-equivariant isometries of $N$ such that $M$ is isometric to $\Gamma \backslash \mkern-5mu \left(S\times_{Z(S)}N\right)$ and $\Gamma$ projects isomorphically to a uniform lattice $\Gamma_0$ in $S/Z(S)$. Also, the centralizer of $\Gamma$ in $\textnormal{Isom}(S \times_{Z(S)} N)$ acts transitively on $S \times_{Z(S)} N$.

$\textnormal{Isom}^0(M)$ is isomorphic to a central quotient of $S \times_{Z(S)} C$, where $C$ is the connected component of the identity in the centralizer of the projection of $\Gamma$ to $Z(S) \cdot \textnormal{Isom}_{Z(S)}(N)$ in $\textnormal{Isom}_{Z(S)}(N)$. $C$ acts transitively on $N$.
\end{enumerate}
\end{theorem}

\begin{remark}
This theorem is \cite[Theorem~6]{G11} (an example of the construction in (ii) is also given there). In particular, one can also find this in \cite[Th{\'e}or{\`e}me~1.7]{Ze98a}. The differences are mainly due to the differences in Theorem~\ref{th:geometric_characterization}. Since the statement about the structure of $\textnormal{Isom}^0(M)$ can be only found in \cite{G11} and will be important in the sequel, we will give a proof of this theorem basing on Theorem~\ref{th:geometric_characterization} in Paragraph~\ref{par:proof_homogeneous}. Also, the proof is more detailed than in \cite{Ze98a}.
\end{remark}

The following theorem is \cite[Theorem~7]{G11}. We will give a proof in Paragraph~\ref{par:reductive}.

\begin{theorem}\label{th:homogeneous_reductive}
Let $M$ be a compact homogeneous Lorentzian manifold and denote by $G:=\textnormal{Isom}^0(M)$ the connected component of the identity in the isometry group. $G$ acts transitively on $M$ and the connected component of the isotropy group $H \subseteq G$ of some point $x \in M$ is compact.

Let $\mathfrak{h}\subseteq \mathfrak{g}$ denote the Lie algebras of both Lie groups. Then, $M\cong G/H$ is reductive, that is, there is an $\textnormal{Ad}(H)$-invariant vector space $\mathfrak{m}\subseteq\mathfrak{g}$ (not necessarily a subalgebra) that is complementary to $\mathfrak{h}$.
\end{theorem}

In the case that the isometry group $\textnormal{Isom}(M)$ has non-compact connected components, we will describe the local geometry of $M$ in Paragraphs~\ref{par:isotropy} and~\ref{par:curvature}. Two of our results \cite[Theorems~8 and~9]{G11} are the following theorems.

\begin{theorem}\label{th:isotropy_representation}
Let $M=(M,g)$ be a compact homogeneous Lorentzian manifold and let ${G:=\textnormal{Isom}^0(M)}$ denote the connected component of the identity in the isometry group. Let $H$ be the isotropy group in $G$ of some $x \in M$. Clearly, $M \cong G/H$.
\begin{enumerate}
\item Suppose the Lie algebra of $G$ contains a direct summand $\mathfrak{s}$ isomorphic to $\mathfrak{sl}_2(\mathds{R})$. Then, the isotropy representation of $G/H$ allows a decomposition into irreducible invariant subspaces such that $\mathfrak{s}$ appears as an irreducible summand.

\item Suppose the Lie algebra of $G$ contains a direct summand $\mathfrak{s}$ isomorphic to $\mathfrak{he}_d^\lambda$, $\lambda \in \mathds{Z}_+^d$. Then, the isotropy representation of $G/H$ allows a decomposition into weakly irreducible invariant subspaces such that $\mathfrak{s}$ appears as a weakly irreducible summand. $\mathfrak{s}$ is not irreducible and cannot be decomposed into a direct sum of irreducible subspaces.
\end{enumerate}
\end{theorem}

\begin{theorem}\label{th:homogeneous_not_Ricci_flat}
Let $M$ be a compact homogeneous Lorentzian manifold that is Ricci-flat. Then, the connected components of its isometry group $\textnormal{Isom}(M)$ are compact.
\end{theorem}

\begin{corollary}\label{cor:Ricci_flat}
Let $M$ be a compact homogeneous Lorentzian manifold that is Ricci-flat, but not flat. Then, its isometry group $\textnormal{Isom}(M)$ is compact.
\end{corollary}
\begin{proof}
Assume that $\textnormal{Isom}(M)$ is not compact. It follows from Theorem~\ref{th:homogeneous_not_Ricci_flat} that $\textnormal{Isom}(M)$ has infinitely many connected components. \cite[Corollary~3]{PiZe10} states that any compact Lorentzian manifold whose isometry group has infinitely many connected components possesses an everywhere timelike Killing vector field if it possesses a somewhere timelike Killing vector field. Since any homogeneous Lorentzian manifold has a somewhere timelike Killing vector field, it follows that there exists an everywhere timelike Killing vector field on $M$. \cite[Theorem~3.2]{RoSa96} yields that any Ricci-flat compact homogeneous Lorentzian manifold admitting a timelike Killing vector field is isometric to a flat torus (up to a finite covering).
\end{proof}
\begin{remark}
Note that in \cite[Paragraph~1]{PiZe10} examples of flat compact homogeneous Lorentzian manifolds are given whose isometry groups are not compact. These manifolds are flat tori provided with the metric defined by certain quadratic forms of Minkowski space.
\end{remark}

%%%%%%%%%%%%%%%%%%%%%%%%%%%%%%%%%%%%%%%%%%%%%%%%%%%%%%%%%%%%%%%%%%%%%%%%%%%%%%%%%%%%%%%%

\section{Compact homogeneous Lorentzian manifolds}\label{sec:homogeneous}

In the following, we will investigate compact homogeneous Lorentzian manifolds, in particular those whose isometry groups have non-compact connected components. We use the geometric description of Theorem~\ref{th:geometric_characterization} to prove Theorem~\ref{th:homogeneous_characterization} in Paragraph~\ref{par:proof_homogeneous}. 

We continue in Paragraph~\ref{par:reductive} with giving a reductive representation for any compact homogeneous Lorentzian manifold. In the case that the connected component of the identity in the isometry group is compact, the statement can be shown in the same way as in the Riemannian case. For the more interesting case of non-compact connected components of the isometry group, the induced bilinear form $\kappa$ yields a reductive representation. Additionally, we show that the isotropy group of a point has compact connected components.

In the case that the isometry group contains a twisted Heisenberg group as a subgroup, we give a different reductive representation in Paragraph~\ref{par:reductive_Heisenberg} that is more convenient with respect to calculation purposes. This representation allows us to determine the local geometry of compact homogeneous Lorentzian manifolds whose isometry groups have non-compact connected components. We will investigate the isotropy representation in Paragraph~\ref{par:isotropy} and the curvatures in Paragraph~\ref{par:curvature}.

%%%%%%%%%%%%%%%%%%%%%%%%%%%%%%%%%%%%%%%%

\subsection{Structure of homogeneous manifolds}\label{par:proof_homogeneous}

Let $M=(M,g)$ be a compact homogeneous Lorentzian manifold. We suppose that $\textnormal{Isom}^0(M)$, the identity component of the isometry group, is not compact. Additionally, let $\mathfrak{isom}(M)=\mathfrak{k}\oplus\mathfrak{a}\oplus\mathfrak{s}$ be the decomposition of its Lie algebra according to Theorem~\ref{th:algebraic_classification}. Due to Proposition~\ref{prop:condition_star} and Theorem~\ref{th:algebraic}, this decomposition is $\kappa$-orthogonal, where $\kappa$ is the induced bilinear form on $\mathfrak{isom}(M)$.

In a homogeneous semi-Riemannian manifold, each tangent vector can be extended to a Killing vector field (cf.~\cite[Corollary~9.38]{ON83}). Thus, the orbits of the isometry group on the Lorentzian manifold $M$ have Lorentzian character. It follows from Theorem~\ref{th:geometric_characterization} that $\mathfrak{s}\cong\mathfrak{sl}_2(\mathds{R})$ or $\mathfrak{s}\cong\mathfrak{he}_d^\lambda$.

Let $S$ be the subgroup generated by $\mathfrak{s}$ and denote by $\mathcal{S}$ the distribution defined by the orbits of $S$. We can apply Theorem~\ref{th:geometric_characterization}~(ii) and~(iii) to see that $M$ is isometric to $\Gamma \backslash \mkern-5mu \left(N \times_\sigma \widetilde{\textnormal{SL}_2(\mathds{R})}\right)$ or $\Gamma \backslash \mkern-5mu \left(S \times_{Z(S)} N\right)$, respectively. $N$ is a Riemannian manifold and $\Gamma$ is a discrete subgroup in ${\textnormal{Isom}(N)\times\widetilde{\textnormal{SL}_2(\mathds{R})}}$ or ${S\times_{Z(S)}\textnormal{Isom}_{Z(S)}(N)}$, respectively.

It is clear from the proof of Theorem~\ref{th:geometric_characterization}~(ii) and~(iii) (see the remark following that theorem) that $N$ is a leaf of the involutive distribution $\mathcal{O}$ ($\mathfrak{s}\cong\mathfrak{sl}_2(\mathds{R})$) or $\mathcal{O}+\mathcal{Z}$ ($\mathfrak{s}\cong\mathfrak{he}_d^\lambda$), respectively. $\mathcal{O}$ is the distribution orthogonal to $\mathcal{S}$, and $\mathcal{Z}$ denotes the orbit of the center of $S$ if $S$ is a twisted Heisenberg group.

Let $H$ be the isotropy group of a point $x_0 \in N \subset M$ in $\textnormal{Isom}^0(M)$. Denote by $\mathfrak{h}$ its Lie algebra.

\begin{proposition}\label{prop:H_compact}
In the situation of this paragraph, the following are true:
\begin{enumerate}
\item Assume that $\mathfrak{s}\cong\mathfrak{sl}_2(\mathds{R})$. Then, $\mathfrak{h} \subseteq \mathfrak{k}\oplus\mathfrak{a}$.

\item Assume that $\mathfrak{s}\cong\mathfrak{he}_d^\lambda$. Then, $\mathfrak{h} \subseteq \mathfrak{k}\oplus\mathfrak{a}\oplus\mathfrak{z(s)}$.
\end{enumerate}
In both cases, it follows that the identity component of $H$ is compact.
\end{proposition}
\begin{proof}
(i) Take $Y=A+W \in \mathfrak{h}$ with $A \in \mathfrak{k}\oplus\mathfrak{a}$, $W \in \mathfrak{s}$. If $\left\{e,f,h\right\}$ is an $\mathfrak{sl}_2$-triple of $\mathfrak{s}$ (fulfilling $[h,e]=2e$, $[h,f]=-2f$, $[e,f]=h$), then $\textnormal{ad}_e$ and $\textnormal{ad}_f$ are nilpotent.

Due to Theorem~\ref{th:locally_free}, the subgroup generated by $\mathfrak{s}$ acts locally freely on $M$. Using $[\mathfrak{s},Y]\subseteq \mathfrak{s}$ and Lemma~\ref{lem:stabilizer} with $X=e$ or $X=f$, we obtain $[e,Y]=0$ as well as $[f,Y]=0$. By the Jacobi identity, $[h,Y]=0$. It follows that $W=0$.

Therefore, $\mathfrak{h}\subseteq\mathfrak{k}\oplus\mathfrak{a}$. Theorem~\ref{th:algebraic} together with Proposition~\ref{prop:condition_star} states that the subgroup generated by $\mathfrak{a}$ is compact. The compact semisimple algebra $\mathfrak{k}$ also generates a compact subgroup. Thus, the subgroup generated by $\mathfrak{h}$ is compact as well, since it is closed.

(ii) Take $Y=A+W \in \mathfrak{h}$ with $A \in \mathfrak{k}\oplus\mathfrak{a}$, $W \in \mathfrak{s}$. Let $X$ be an arbitrary element of the nilradical of $\mathfrak{s}$, which is isomorphic to $\mathfrak{he}_d$.

Due to Theorem~\ref{th:locally_free}, the subgroup generated by $\mathfrak{s}$ acts locally freely on $M$. Using $[\mathfrak{s},Y]\subseteq \mathfrak{s}$ and Lemma~\ref{lem:stabilizer}, we obtain $[X,Y]=0$. Since this is true for any $X$ as above, $W$ centralizes the nilradical of $\mathfrak{s}$. It follows that $W \in \mathfrak{z(s)}$ because $\mathfrak{s}$ is a twisted Heisenberg algebra.

Therefore, $\mathfrak{h}\subseteq\mathfrak{k}\oplus\mathfrak{a}\oplus\mathfrak{z(s)}$. Theorem~\ref{th:algebraic} together with Proposition~\ref{prop:condition_star} states that the subgroup generated by $\mathfrak{a}\oplus\mathfrak{z(s)}$ is compact. We can now conclude exactly as in~(i) that $\mathfrak{h}$ generates a compact subgroup.
\end{proof}

\begin{proposition}\label{prop:N_compact}
In the situation of this section, $N$ is a compact homogeneous Riemannian manifold.
\end{proposition}
\begin{proof}

Consider the decomposition $\mathfrak{isom}(M)=\mathfrak{k}\oplus\mathfrak{a}\oplus\mathfrak{s}$, where $\mathfrak{s}$ is isomorphic to $\mathfrak{sl}_2(\mathds{R})$ or $\mathfrak{he}_d^\lambda$, respectively. Let $C$ be the subgroup generated by \[\mathfrak{c}:=\mathfrak{k}\oplus\mathfrak{a} \text{ or } \mathfrak{c}:=\mathfrak{k}\oplus\mathfrak{a}\oplus\mathfrak{z(s)},\] respectively. But $\mathfrak{a}$ and $\mathfrak{a}\oplus\mathfrak{z(s)}$, respectively, generate compact subgroups by Proposition~\ref{prop:condition_star} and Theorem~\ref{th:algebraic}. $\mathfrak{k}$ is compact semisimple and also generates a compact subgroup. Thus, $C$ is compact.

The centralizer of $S$ in $\textnormal{Isom}^0(M)$ is a Lie subgroup whose Lie algebra coincides with the centralizer of $\mathfrak{s}$ in $\mathfrak{isom}(M)$. But the centralizer of $\mathfrak{s}$ is given by $\mathfrak{c}$, therefore, $C$ is contained in the centralizer of $S$.

Let $f \in C$. Using $ff^\prime \cdot x=f^\prime f \cdot x$ for all $x \in M$ and $f^\prime \in S$, we see that $f$ preserves the orbits $\mathcal{S}$ of $S$. But $f$ is an isometry, so $\mathcal{O}=\mathcal{S}^\perp$ is preserved as well. Restricting $f^\prime$ to elements of the center of $S$, we also obtain that $f$ preserves $\mathcal{Z}$ in the case $\mathfrak{s}\cong\mathfrak{he}_d^\lambda$. Thus, for all $f \in C$, $f \cdot N$ is a leaf of the foliation induced by the involutive distribution $\mathcal{O}$ or $\mathcal{O}+\mathcal{Z}$, respectively.

If $f \in S$, then $\mathcal{O}$ is obviously preserved by $f$. Using that $Z(S)$ is central, we see that $f$ preserves $\mathcal{Z}$ in the case $\mathfrak{s}\cong\mathfrak{he}_d^\lambda$ as well: For a generator $Z$ of $Z(S)$, \begin{align*}
df_x(\widetilde{Z}(x))&=df_x(\frac{\partial}{\partial t} \left(\exp(tZ) \cdot x\right)|_{t=0})\\
&=\frac{\partial}{\partial t} \left(f\exp(tZ) \cdot x\right)|_{t=0}\\
&=\frac{\partial}{\partial t}\left( \exp(tZ) \cdot (f\cdot x)\right)|_{t=0}\\
&=\widetilde{Z}(f \cdot x).\end{align*} As above, $f \cdot N$ is a leaf for all $f \in S$.

Because $\mathfrak{isom}(M)=\mathfrak{c}+\mathfrak{s}$, it follows that $f \cdot N$ is a leaf for all $f \in \textnormal{Isom}^0(M)$.

Let $\widehat{C}^0$ be the identity component of the subgroup $\widehat{C}$ in $\textnormal{Isom}^0(M)$ that maps the leaf $N$ to itself. Since $M$ is homogeneous, for any $x,y \in N$ there is $f \in \textnormal{Isom}(M)$ such that $f \cdot x=y$. But $f \cdot N$ is also a leaf, hence $f \cdot N=N$ and $f \in \widehat{C}$. It follows that $N$ is a homogeneous space. Since $N$ is connected, $\widehat{C}^0$ acts transitively on $N$.

Let $\widehat{\mathfrak{c}}$ be the Lie algebra of $\widehat{C}$. By definition, any Killing vector field generated by an element of $\widehat{\mathfrak{c}}$ has to lie in $\mathcal{O}$ or $\mathcal{O}+\mathcal{Z}$, respectively. But $\mathcal{O}=\mathcal{S}^\perp$ and $\mathcal{Z}$ is the orbit of the center of $S$ if it is a twisted Heisenberg group. Thus, $\widehat{\mathfrak{c}}$ is $\kappa$-orthogonal to $\mathfrak{s}$ in the case $\mathfrak{s}\cong\mathfrak{sl}_2(\mathds{R})$ and $\widehat{\mathfrak{c}}\subseteq\mathfrak{o}+\mathfrak{z(s)}$ if $\mathfrak{s}\cong\mathfrak{he}_d^\lambda$, where $\mathfrak{o}$ is the $\kappa$-orthogonal complement of $\mathfrak{s}$ in $\mathfrak{isom}(M)$. It follows that \[\widehat{\mathfrak{c}}\subseteq \mathfrak{c} \text{ and }\widehat{C}^0\subseteq C.\]

We want to show that $\widehat{C}^0= C$. Suppose the contrary. Then for any neighborhood $U$ of the identity in $C$, there is $f \in U$ such that $f \cdot N$ is not equal to $N$.

Let $U^\prime$ be a neighborhood of the identity in $S$. Since the action of $S$ on $M$ is locally free by Theorem~\ref{th:locally_free} and $\mathcal{O}=\mathcal{S}^\perp$, any leaf is locally given by $f^\prime \cdot N$, $f^\prime \in U^\prime$.

If the neighborhood $U$ is small enough, then for any $f \in U$ such that $f \cdot N$ is not equal to $N$, there is $f^\prime \in U^\prime$ such that $f \cdot N=f^\prime \cdot N$.

Since $\widehat{C}^0$ acts transitively on $N$, the map $\widehat{C}^0 \to N$, $\widehat{f} \mapsto \widehat{f} \cdot x_0$, is a locally trivial fiber bundle. So if we fix a neighborhood $\widehat{U}$ of the identity in $\widehat{C}^0$, $\widehat{U} \cdot x_0$ is a neighborhood of $x_0$ in $N$. It follows that if $U^\prime$ and $U$ are sufficiently small, we can choose $f \in U$, $f^\prime \in U^\prime$ such that \[f \cdot N=f^\prime \cdot N\neq N,\] and there is $\widehat{f} \in \widehat{U}$ such that \[\widehat{f}f^{-1}f^\prime \cdot x_0=x_0,\] that is, \[\widehat{f}f^{-1}f^\prime \in H.\] If the neighborhoods $\widehat{U}$, $U^\prime$ and $U$ are sufficiently small, even \[\widehat{f}f^{-1}f^\prime \in H^0,\] where $H^0$ is the identity component in $H$.

By Proposition~\ref{prop:H_compact}, $\mathfrak{h}\subseteq \mathfrak{c}$. Therefore, $H^0 \subseteq C$. Thus, \[\widehat{f}f^{-1}f^\prime \in C.\] Using that $\widehat{f} \in \widehat{C}^0 \subseteq C$ and $f \in U \subset C$, it follows that \[f^\prime \in C.\] If $U$ is small enough, it follows from $\mathfrak{c} \cap \mathfrak{s}=\mathfrak{z(s)}$ that $f^\prime$ is contained in the subgroup generated by $\mathfrak{z(s)}$. But the latter group preserves $N$, contradicting $f^\prime \cdot N\neq N$. Hence, $\widehat{C}^0= C$ and $\widehat{C}^0$ is compact. It follows that $N$ is compact as well.
\end{proof}

Let $\mathfrak{s}\cong\mathfrak{sl}_2(\mathds{R})$. Since $C$ acts transitively on $N$, it follows that $\sigma$ has to be constant. Therefore, $M$ is covered isometrically by the metric product $N \times \widetilde{\textnormal{SL}_2(\mathds{R})}$, where $\widetilde{\textnormal{SL}_2(\mathds{R})}$ is furnished with the metric defined by a positive multiple of the Killing form of $\mathfrak{sl}_2(\mathds{R})$.

Assume that $\mathfrak{s}\cong\mathfrak{he}_d^\lambda$. In the same way as above, we see that the map $m:N \to \mathcal{M}$ is a constant, which we denote now by the ad-invariant Lorentz form $m$.

The proof of the last proposition shows even more:

\begin{corollary}\label{cor:isometry_group}
In the situation of this paragraph, the following are true:
\begin{enumerate}
\item Let $s\cong\mathfrak{sl}_2(\mathds{R})$. Then, $\textnormal{Isom}^0(M)$ is isomorphic to a central quotient of the group $C \times \widetilde{\textnormal{SL}_2(\mathds{R})}$. Here, $C$ is the identity component of the centralizer of the projection of $\Gamma$ to $\textnormal{Isom}(N)$. $C$ acts transitively on $N$.

\item Let $\mathfrak{s}\cong\mathfrak{he}_d^\lambda$. Then, $\textnormal{Isom}^0(M)$ is isomorphic to a central quotient of the group $S \times_{Z(S)} C$. Here, $C$ is the identity component of the centralizer of the projection of $\Gamma$ to $Z(S)\cdot\textnormal{Isom}_{Z(S)}(N)$. $C$ acts transitively on $N$.
\end{enumerate}
\end{corollary}
\begin{proof}
We have seen in the proof of Proposition~\ref{prop:N_compact} that the compact subgroup $C$ generated by $\mathfrak{c}:=\mathfrak{k}\oplus\mathfrak{a}$ or $\mathfrak{c}:=\mathfrak{k}\oplus\mathfrak{a}\oplus\mathfrak{z(s)}$, respectively, preserves the leaf $N$ and acts transitively on $N$. Therefore, we can see $C$ as a subgroup of $\textnormal{Isom}(N)$. In case~(ii), $C$ is even a subgroup of $\textnormal{Isom}_{Z(S)}(N)$, since $\mathfrak{z(s)}$ is central in $\mathfrak{c}$, so any isometry in $C$ commutes with the action of $Z(S)$.

Because of $M \cong \Gamma \backslash \mkern-5mu \left(N \times \widetilde{\textnormal{SL}_2(\mathds{R})}\right)$ or $M \cong \Gamma \backslash \mkern-5mu \left(S \times_{Z(S)} N\right)$, respectively, an element of $\textnormal{Isom}(N)$ or $\textnormal{Isom}_{Z(S)}(N)$, respectively, induces an isometry of $M$ if and only if it commutes with the projection of $\Gamma$ to $\textnormal{Isom}(N)$ or $Z(S)\cdot\textnormal{Isom}_{Z(S)}(N)$, respectively. Note that due to Theorem~\ref{th:geometric_characterization}, $\Gamma$ is a discrete subgroup of the group $\text{Isom}(N) \times \widetilde{\textnormal{SL}_2(\mathds{R})}$ or $S \times_{Z(S)} \text{Isom}_{Z(S)}(N)$, respectively. Here, the projection to $Z(S)\cdot\textnormal{Isom}_{Z(S)}(N)$ is given by \[[f,\psi] \mapsto Z(S) \cdot \psi \subseteq \text{Isom}_{Z(S)}(N),\] $[f,\psi]$ is the equivalence class of $(f,\psi)\in S \times \text{Isom}_{Z(S)}(N)$ in $S \times_{Z(S)} \text{Isom}_{Z(S)}(N)$.

In case~(ii), $Z(S)$ is contained in $S$ and $C$ and acts as the same group of isometries on $M$. Thus, $S \times_{Z(S)} C$ is defined.

The Lie algebra of $C \times S$ or $S \times_{Z(S)} C$, respectively, is isomorphic to the Lie algebra of the isometry group,  $\mathfrak{isom}(M)=\mathfrak{k}\oplus\mathfrak{a}\oplus\mathfrak{s}$.

The action of $C \times S$ or $S \times_{Z(S)} C$, respectively, on $M$ is clearly locally effective. Additionally, $S$ and $C$ are Lie subgroups of $\textnormal{Isom}^0(M)$. Therefore, the canonical Lie group homomorphism $C \times S \to \textnormal{Isom}^0(M)$ or $S \times_{Z(S)} C \to \textnormal{Isom}^0(M)$, respectively, is surjective and the kernel is a discrete central subgroup of $C \times S$ or $S \times_{Z(S)} C$, respectively.

Finally, the covering $\widetilde{\textnormal{SL}_2(\mathds{R})} \to S$ is central in case~(i).
\end{proof}

Consider the canonical projection $P:\textnormal{Isom}(N) \times \widetilde{\textnormal{SL}_2(\mathds{R})} \to \widetilde{\textnormal{SL}_2(\mathds{R})}$ (projection to the second component) or $P:S \times_{Z(S)} \textnormal{Isom}_{Z(S)}(N) \to S/Z(S)$ (projection to the first component), respectively. In both situations, $P$ is a continuous homomorphism that is surjective and has kernel $\textnormal{Isom}(N)$ or kernel isomorphic to $\textnormal{Isom}_{Z(S)}(N)$.

Since $N$ is a compact Riemannian manifold by Proposition~\ref{prop:N_compact}, $\textnormal{Isom}(N)$ is compact. $\textnormal{Isom}_{Z(S)}(N)$ as a closed subgroup of $\textnormal{Isom}(N)$ is compact as well. Thus, in any case, $P$ has compact kernel. In the following lemma, we show that $\Gamma_0:=P(\Gamma)$ is discrete.

\begin{lemma}\label{lem:compact_kernel_discrete_subgroup}
Let $G$ and $G^\prime$ be connected Lie groups and $P:G \to G^\prime$ be a surjective continuous homomorphism with compact kernel. If $\Gamma \subset G$ is discrete, $P(\Gamma) \subset G^\prime$ is discrete as well.
\end{lemma}
\begin{proof}
Assume the contrary. Then, there is a sequence $\left\{f_k\right\}_{k=0}^\infty$ of distinct elements in $\Gamma$ such that $P(f_k)$ converges to $P(f)$ in $G^\prime$ as $k \to \infty$ for some $f \in G$.

Since $G \to G^\prime$ is a locally trivial fiber bundle with fiber $\textnormal{ker} (P)$, there is a sequence $\left\{p_k\right\}_{k=0}^\infty$ in $\textnormal{ker} (P)$ such that $f_k p_k \to f$ converges in $G$ as $k \to \infty$. But $\textnormal{ker} (P)$ is compact, hence, we may assume already that $p_k \to p$ converges. It follows that $f_k \to f p^{-1}$ as $k \to \infty$, contradicting the fact that $\Gamma$ is discrete.
\end{proof}

\begin{lemma}\label{lem:gamma_0_cocompact}
$\Gamma_0$ is cocompact in $\widetilde{\textnormal{SL}_2(\mathds{R})}$ ($\mathfrak{s}\cong\mathfrak{sl}_2(\mathds{R})$) or $S/Z(S)$ ($\mathfrak{s}\cong\mathfrak{he}_d^\lambda$), respectively. Thus, $\Gamma_0$ is a uniform lattice.
\end{lemma}
\begin{proof}
Let $\widetilde{M}:=N \times \widetilde{\textnormal{SL}_2(\mathds{R})}$ if $\mathfrak{s}\cong\mathfrak{sl}_2(\mathds{R})$ and $\widetilde{M}:=S \times_{Z(S)} N$ if $\mathfrak{s}\cong\mathfrak{he}_d^\lambda$. Consider the covering map $\pi: \widetilde{M} \to M$ and cover $M$ by evenly covered precompact open sets. Since $M$ is compact, we may choose finitely many of them such that they still cover $M$. For any such precompact open set, choose one sheet in $\widetilde{M}$. The collection of these sheets is precompact and we denote the closure of them by $A$. Because of $M \cong \Gamma \backslash \widetilde{M}$, the deck transformations are acting transitively on the fibers. Hence, $\Gamma \cdot A=\widetilde{M}$.

Let $B$ denote the projection of $A$ under the projection $\widetilde{M} \to S_0$, where $S_0:=\widetilde{\textnormal{SL}_2(\mathds{R})}$ if $\mathfrak{s}\cong\mathfrak{sl}_2(\mathds{R})$ and $S_0:=S/Z(S)$ if $\mathfrak{s}\cong\mathfrak{he}_d^\lambda$. $B$ is compact and $B \cdot\Gamma_0 = S_0$.

We want to show that $S_0/\Gamma_0$ is compact. Equivalently, we have to show that any sequence $\left\{f_k\Gamma_0\right\}_{k=0}^\infty$, $f_k \in S_0$, has a convergent subsequence. For all $k$, choose $b_k \in B$ and $\gamma_k \in \Gamma_0$ such that $b_k \gamma_k=f_k$. $B$ is compact, so we may choose a convergent subsequence $\left\{b_{k_j}\Gamma_0\right\}_{j=0}^\infty$. Therefore, $f_{k_j}\Gamma_0=b_{k_j}\Gamma_0$ converges as $j \to \infty$.
\end{proof}

The kernel of the projection $\Gamma \to \Gamma_0$ is (isomorphic to) a discrete subgroup of $\textnormal{Isom}(N)$ or $\textnormal{Isom}_{Z(S)}(N)$, respectively. Note that the kernel is finite, since the latter groups are compact. Any element of the kernel can be written as $(\psi,e) \in \textnormal{Isom}(N) \times \widetilde{\textnormal{SL}_2(\mathds{R})}$ or $[e,\psi] \in S \times_{Z(S)} \textnormal{Isom}_{Z(S)}(N)$, respectively, $e$ being the identity element of the corresponding group.

Consider now the covering map $\pi: N \times \widetilde{\textnormal{SL}_2(\mathds{R})} \to M$ or $\pi: S \times_{Z(S)} N \to M$, respectively. As remarked after Theorem~\ref{th:geometric_characterization}, these maps are canonically defined through the action of $S$ on $M$ (note that in the case $\mathfrak{s}\cong\mathfrak{sl}_2(\mathds{R})$, $\widetilde{\textnormal{SL}_2(\mathds{R})}$ covers $S$). By definition, \[\psi(x_0)=\pi(\psi(x_0),e)=\pi((\psi,e)(x_0,e))=\pi(x_0,e)=x_0\] in the first case. Analogously, \[\psi(x_0)=\pi([e,\psi(x_0)])=\pi([e,\psi]([e,x_0]))=\pi([e,x_0])=x_0\] in the second case.
It follows that $(\psi,e)$ and $[e,\psi]$, respectively, have the fixed point $(x_0,e)$ on $N \times \widetilde{\textnormal{SL}_2(\mathds{R})}$ or $[e,x_0]$ on $S \times_{Z(S)} N$, respectively. But $\Gamma$ acts freely, so it has to be the identity element. Thus, the kernel of the projection $\Gamma \to \Gamma_0$ is trivial.

It follows that $\Gamma$ projects isomorphically to $\Gamma_0$. In case~(i) of the theorem, we obtain that $\Gamma$ is the graph of a homomorphism $\varrho:\Gamma_0 \to \textnormal{Isom}(N)$.

$\Gamma$ corresponds to the group of deck transformations. Since $M$ is homogeneous, the centralizer of $\Gamma$ in the isometry group of the covering manifold acts transitively (cf.~\cite[Theorem~2.5]{Wo61}). This completes the proof of Theorem~\ref{th:homogeneous_characterization}.

%%%%%%%%%%%%%%%%%%%%%%%%%%%%%%%%%%%%%%%%

\subsection{Reductivity}\label{par:reductive}

Let $M$ be a compact homogeneous Lorentzian manifold and $G:=\textnormal{Isom}^0(M)$. $G$ acts transitively on the Lorentzian manifold $M$. We consider the isotropy group $H \subseteq G$ of some point $x \in M$. Denote by $\mathfrak{h}\subseteq\mathfrak{g}$ the corresponding Lie algebras.

If $G$ is compact, $H$ is compact as well, since $H$ is a closed subgroup. Therefore, $\mathfrak{g}$ possesses an ad-invariant symmetric bilinear form $b$ that is positive definite. $G$ is connected, therefore, $b$ is $\textnormal{Ad}(G)$-invariant. Now choose $\mathfrak{m}$ to be the $b$-orthogonal complement to $\mathfrak{h}$ in $\mathfrak{g}$. $\mathfrak{m}$ is $\textnormal{Ad}(H)$-invariant, because $\mathfrak{h}$ and $b$ are.

Suppose in the following that $G$ is not compact. By Proposition~\ref{prop:H_compact}, the connected components of $H$ are compact. To finish the proof of Theorem~\ref{th:homogeneous_reductive}, we have to find a reductive representation of $M$.

According to Theorems~\ref{th:algebraic_classification} and~\ref{th:homogeneous_characterization}, we have a decomposition $\mathfrak{g}=\mathfrak{k}\oplus\mathfrak{a}\oplus\mathfrak{s}$, where $\mathfrak{k}$ is compact semisimple, $\mathfrak{a}$ is abelian, and $\mathfrak{s}$ is either isomorphic to $\mathfrak{sl}_2(\mathds{R})$ or to $\mathfrak{he}_d^\lambda$, $\lambda \in \mathds{Z}_+^d$. Moreover, the induced bilinear form $\kappa$ on $\mathfrak{g}$ is Lorentzian by Proposition~\ref{prop:condition_star} and Theorem~\ref{th:algebraic}.

$\kappa$ is ad-invariant and $G$ is connected, so $\kappa$ is $\textnormal{Ad}(G)$-invariant. Since the subgroup generated by $\mathfrak{s}$ acts locally freely on $M$ by Theorem~\ref{th:locally_free}, $\mathfrak{s}\cap\mathfrak{h}=\left\{0\right\}$. It follows from Theorem~\ref{th:algebraic} and Proposition~\ref{prop:H_compact} that $\kappa$ restricted to $\mathfrak{h} \times \mathfrak{h}$ is positive definite. For this, note that $\mathfrak{z(s)}$ is $\kappa$-isotropic in the case $\mathfrak{s}\cong\mathfrak{he}_d^\lambda$.

Choosing $\mathfrak{m}$ to be the $\kappa$-orthogonal complement of $\mathfrak{h}$ in $\mathfrak{g}$, we are done.

\subsubsection{Isometry group contains a twisted Heisenberg group}\label{par:reductive_Heisenberg}

We now additionally assume that the Lie algebra of $G:=\textnormal{Isom}^0(M)$ contains a direct summand $\mathfrak{s}$ isomorphic to $\mathfrak{he}_d^\lambda$, $\lambda \in \mathds{Z}_+^d$. In the following, we will give a reductive representation that is more convenient with respect to calculation purposes than the one given before.

According to Proposition~\ref{prop:condition_star} and Theorem~\ref{th:algebraic}, the Lie algebra of $G$ decomposes as a $\kappa$-orthogonal direct sum $\mathfrak{g}=\mathfrak{s}\oplus\mathfrak{k}\oplus\mathfrak{a}$, $\kappa$ being the induced bilinear form on $\mathfrak{g}$. $\mathfrak{a}$ is abelian and $\mathfrak{k}$ is semisimple. Additionally, $\kappa$ restricted to $\mathfrak{s} \times \mathfrak{s}$ is an ad-invariant Lorentzian scalar product on $\mathfrak{s}$ and the restriction to $(\mathfrak{k}\oplus\mathfrak{a})\times(\mathfrak{k}\oplus\mathfrak{a})$ is positive definite.

Due to Proposition~\ref{prop:H_compact}, it holds for the Lie algebra $\mathfrak{h}$ of $H$ that $\mathfrak{h}\subseteq \mathfrak{z(s)}\oplus\mathfrak{k}\oplus\mathfrak{a}$. By Theorem~\ref{th:locally_free}, $\mathfrak{s}\cap\mathfrak{h}=\left\{0\right\}$. Thus, we can consider the $\kappa$-orthogonal complement $\mathfrak{m}^\prime$  of $\mathfrak{h}$ in $\mathfrak{z(s)}\oplus\mathfrak{k}\oplus\mathfrak{a}$. Then, $\mathfrak{m}^\prime$ is $\textnormal{Ad}(H)$-invariant. Note that $\mathfrak{z(s)}\subseteq \mathfrak{m}^\prime$.

$\mathfrak{s}$ is an ideal in $\mathfrak{g}$ and hence $\textnormal{Ad}(G)$-invariant. It follows that $\mathfrak{m}:=\mathfrak{s}+\mathfrak{m}^\prime$ is $\textnormal{Ad}(H)$-invariant and complementary to $\mathfrak{h}$. Thus, $\mathfrak{g}=\mathfrak{m}\oplus \mathfrak{h}$ is a reductive decomposition.

By Theorem~\ref{th:homogeneous_characterization}, $M$ is isometric to $\Gamma\backslash\mkern-5mu\left(S \times_{Z(S)} N\right)$, where $N$ is a compact homogeneous Riemannian manifold and $S$ is provided with the metric defined by an ad-invariant Lorentzian scalar product on $\mathfrak{s}$. Also, $G$ is a central quotient of $S \times_{Z(S)} C$, where $C$ is a subgroup of the isometry group of $N$ acting transitively on $N$. It follows from the proof of Corollary~\ref{cor:isometry_group} that $C$ can be identified with the subgroup in $G$ generated by $\mathfrak{c}:=\mathfrak{z(s)}\oplus\mathfrak{k}\oplus\mathfrak{a}$.

Since $C \subset G$, the isotropy group $H_C$ of the point $x$ in $N$ (remember that we can consider $N$ as a leaf in $M$; see the remark following Theorem~\ref{th:geometric_characterization}) in the group $C$ is contained in $H$. Moreover, since $\mathfrak{h}\subseteq \mathfrak{c}$, it follows that the Lie algebra of $H_C$ is $\mathfrak{h}$ as well. Thus, $N \cong C/H_C$ and $\mathfrak{c}=\mathfrak{m}^\prime\oplus\mathfrak{h}$ is a reductive decomposition. Note that the $\textnormal{Ad}(H_C)$-invariance of $\mathfrak{m}^\prime$ follows from the fact that the latter space is the $\kappa$-orthogonal complement of $\mathfrak{h}$.

The metric on $N$ corresponds to an $\textnormal{Ad}(H_C)$-invariant Riemannian scalar product $(\cdot,\cdot)$ on $\mathfrak{m}^\prime$. Let $\mathfrak{p}$ be the $(\cdot,\cdot)$-orthogonal complement of $\mathfrak{z(s)}$ in $\mathfrak{m}^\prime$. We obtain $\mathfrak{m}^\prime=\mathfrak{z(s)}\oplus\mathfrak{p}$ and $\mathfrak{m}=\mathfrak{s}\oplus\mathfrak{p}$.

Let $\langle \cdot,\cdot \rangle$ denote the Lorentzian scalar product on $\mathfrak{m}$ corresponding to the metric on $M$. It follows from the construction that the direct sum $\mathfrak{m}=\mathfrak{s}\oplus\mathfrak{p}$ is $\langle \cdot,\cdot \rangle$-orthogonal and $\langle \cdot,\cdot \rangle$ restricted to $\mathfrak{s} \times \mathfrak{s}$ is an ad-invariant Lorentzian scalar product. The restriction to $\mathfrak{p} \times \mathfrak{p}$ is Riemannian and equals $(\cdot,\cdot)$.

Note that $S$ furnished with the metric defined by the restriction of $\langle \cdot, \cdot \rangle$ to $\mathfrak{s} \times \mathfrak{s}$ is a symmetric space.

%%%%%%%%%%%%%%%%%%%%%%%%%%%%%%%%%%%%%%%%

\subsection{Isotropy representation}\label{par:isotropy}

Let $M$ be a compact homogeneous Lorentzian manifold and $G:=\textnormal{Isom}^0(M)$. $G$ acts transitively on the Lorentzian manifold $M$. We consider the isotropy group $H \subseteq G$ of some point $x \in M$. Denote by $\mathfrak{h}\subseteq\mathfrak{g}$ the corresponding Lie algebras. We will use the reductive decomposition $\mathfrak{g}=\mathfrak{m}\oplus\mathfrak{h}$ given in Paragraph~\ref{par:reductive} (for the case that $\mathfrak{g}$ contains a direct factor isomorphic to $\mathfrak{he}_d^\lambda$, we use the decomposition given in Paragraph~\ref{par:reductive_Heisenberg}). Note that $T_{x}M$ can be identified with $\mathfrak{m}$.

\begin{definition}
For a homogeneous space $M=G/H$, the map $H \to \textnormal{GL}(T_{x} M)$, $h \to dh_{x}$, is called the \textit{isotropy representation} of $G/H$. $G/H$ is called \textit{isotropy irreducible} if the isotropy representation is irreducible, that is, there are no other $H$-invariant subspaces in $T_{x} M$ than $\left\{0\right\}$ and $T_{x} M$. $G/H$ is called \textit{weakly isotropy irreducible} if any $H$-invariant subspace in $T_{x} M$ is either trivial or degenerate (with respect to the metric $g_{x}$).
\end{definition}

To prove Theorem~\ref{th:isotropy_representation}, we have to distinguish the cases whether $\mathfrak{g}$ contains a direct factor $\mathfrak{s}\cong\mathfrak{sl}_2(\mathds{R})$ or $\mathfrak{s}\cong\mathfrak{he}_d^\lambda$. Since the ideas of proof are very similar in both cases, we will here only look at the slightly harder case $\mathfrak{s}\cong\mathfrak{he}_d^\lambda$. For the case $\mathfrak{s}\cong\mathfrak{sl}_2(\mathds{R})$, one can find the proof in \cite[Proposition~5.12]{G11}.

\begin{proposition}\label{prop:heis_isotropy}
Let $M$ be a compact homogeneous Lorentzian manifold and denote $G:=\textnormal{Isom}^0(M)$. We consider the isotropy group $H \subseteq G$ of some point $x \in M$. Denote by $\mathfrak{h}\subseteq\mathfrak{g}$ the corresponding Lie algebras. Assume that $\mathfrak{g}$ contains a direct factor $\mathfrak{s}\cong\mathfrak{he}_d^\lambda$. We will use the reductive decomposition $\mathfrak{g}=\mathfrak{m}\oplus\mathfrak{h}$ given in Paragraph~\ref{par:reductive_Heisenberg} (the decomposition $\mathfrak{m}=\mathfrak{s}\oplus\mathfrak{p}$ is $\langle \cdot,\cdot \rangle$-orthogonal, where the $\textnormal{Ad}(H)$-invariant scalar product $\langle \cdot,\cdot \rangle$ corresponds to the metric on $M$).

Then, a decomposition of $\mathfrak{m}$ into a direct sum of weakly irreducible $\textnormal{Ad}(H)$-invariant subspaces is given by the $\langle \cdot,\cdot \rangle$-orthogonal sum $\mathfrak{s}\oplus\mathfrak{p}_1\oplus\ldots\oplus\mathfrak{p}_k$, where the $\mathfrak{p}_j$ are irreducible subspaces of $\mathfrak{p}$. Furthermore, $\mathfrak{s}$ is not irreducible, but cannot be decomposed into irreducible invariant subspaces.

In particular, the space $G/H$ is not isotropy irreducible. It is weakly isotropy irreducible if and only if $M\cong S/\Gamma$.
\end{proposition}
\begin{proof}
By Theorem~\ref{th:homogeneous_characterization}, there is a uniform lattice $\Gamma_0$ in $S/Z(S)$ such that $\Gamma$ projects isomorphically to $\Gamma_0$. Remember that $[\gamma,\psi_\gamma] \in \Gamma \subset S \times_{Z(S)} \textnormal{Isom}_{Z(S)}(N)$ is acting on $S \times_{Z(S)} N$ by \[[\gamma,\psi_\gamma]([f,x])\mapsto [f\gamma,\psi_\gamma(x)].\] We remarked this after Theorem~\ref{th:geometric_characterization}. Denote by $\pi: S \times_{Z(S)} N \to M$ the covering given by the theorem. Let $e$ be the identity element of $S$. Choose $\gamma \in \Gamma_0 \cdot Z(S) \subset S$ arbitrary and $\psi_\gamma$ in $\textnormal{Isom}_{Z(S)}(N)$ such that $[\gamma,\psi_\gamma] \in \Gamma$.

Since $C$ acts transitively on $N$, there is $\psi \in C$ such that $\psi_\gamma(x)=\psi(x)$. Consider now the map $\varphi_{\gamma}: M \to M$ defined by \[\varphi_{\gamma}(\pi([f,y])):=\pi[(f\gamma,\psi(y))].\] Due to Theorem~\ref{th:homogeneous_characterization}, $C\subseteq \textnormal{Isom}_{Z(S)}(N)$ centralizes $\Gamma$, therefore, $\varphi_{\gamma}$ is correctly defined. By construction, $\varphi_{\gamma}$ is an isometry and \[x=\pi([e,x])=\pi([\gamma,\psi_\gamma(x)])=\pi([\gamma,\psi(x)]),\] that is, $\varphi_{\gamma} \in H$.

$\mathfrak{s}$ is an ideal in $\mathfrak{g}$, in particular $\textnormal{Ad}(H)$-invariant. $\mathfrak{p}$ is the $\langle \cdot,\cdot \rangle$-orthogonal complement of $\mathfrak{s}$ in $\mathfrak{m}$, so it is $\textnormal{Ad}(H)$-invariant as well. Thus, $\mathfrak{m}=\mathfrak{s}\oplus\mathfrak{p}$ is a decomposition into invariant subspaces. Since the restriction of $\langle \cdot,\cdot \rangle$ to $\mathfrak{p}$ is positive definite, we obtain by induction an orthogonal decomposition $\mathfrak{p}=\mathfrak{p}_1\oplus\ldots\oplus\mathfrak{p}_k$ into invariant irreducible subspaces. It remains to show that $\mathfrak{s}$ is weakly irreducible.

Because $C$ centralizes $\mathfrak{s}$, the adjoint action of $C$ on $\mathfrak{s}$ is trivial. Hence, $\textnormal{Ad}_{\varphi_{\gamma}}|_{\mathfrak{s}}=\textnormal{Ad}^S_{\gamma}$, $\textnormal{Ad}^S$ being the adjoint action of $S$. Since the adjoint action of the center is trivial, $S/Z(S)$ acts on $\mathfrak{s}$ through the adjoint action as well. We denote this action by $\textnormal{Ad}^{S/Z(S)}$.

Thus, any $\textnormal{Ad}(H)$-invariant subspace of $\mathfrak{s}$ is also $\textnormal{Ad}^{S/Z(S)}(\Gamma_0)$-invariant.

Let $N$ denote the nilradical of $S/Z(S)$. $N$ is isomorphic to $\widetilde{\textnormal{He}_d}/Z(\widetilde{\textnormal{He}_d})\cong \mathds{R}^{2d}$. Moreover, $\Gamma_0 \cap N$ is a lattice in $N$ (cf.~\cite[Part~I, Chapter~2, Theorem~3.6]{OnVin00}), and $\Gamma_0 \cap N$ is Zariski-dense in $N$ since $N$ is simply-connected and nilpotent (cf.~\cite[Part~I, Chapter~2, Theorem~2.4]{OnVin00}). It follows that any subspace of $\mathfrak{s}$ invariant under $\textnormal{Ad}^{S/Z(S)}(\Gamma_0)$ is $\textnormal{ad}(\mathfrak{he}_d)$-invariant.

Let $\mathfrak{i} \subseteq \mathfrak{s}$ be a non-trivial $\textnormal{ad}(\mathfrak{he}_d)$-invariant subspace. If $\mathfrak{i} \subseteq \mathfrak{he}_d$, it follows that $Z \in \mathfrak{i}$ and $\mathfrak{i}$ is degenerate. Otherwise, $T+X \in \mathfrak{i}$, $X \in \mathfrak{he}_d$. But $[\mathfrak{he}_d,T]$ is equal to the subspace generated by $X_1,Y_1,\ldots,X_d,Y_d$. Thus, all $X_k$ and $Y_k$, $k=1,\ldots, d$, are contained in $\mathfrak{i} + \mathds{R}Z$. Applying $\textnormal{ad}(\mathfrak{he}_d)$-invariance another time, we obtain $Z \in \mathfrak{i}$, so $\mathfrak{i}=\mathfrak{s}$. It follows that $\mathfrak{s}$ is weakly irreducible.

Since $\mathfrak{he}_d$ is an ideal in $\mathfrak{g}$, it is $\textnormal{Ad}(H)$-invariant. Thus, $\mathfrak{s}$ is not irreducible. We have seen above that any non-trivial $\textnormal{Ad}(H)$-invariant subspace $\mathfrak{i}\subseteq \mathfrak{s}$ contains $\mathfrak{z(s)}$. Therefore, $\mathfrak{s}$ cannot be decomposed into irreducible invariant subspaces.
\end{proof}
\begin{remark}
Note that the subspaces $\mathfrak{p_j}$ in Proposition~\ref{prop:heis_isotropy} are $\textnormal{Ad}(H_C)$-invariant, since $H_C \subseteq H$.
\end{remark}

%%%%%%%%%%%%%%%%%%%%%%%%%%%%%%%%%%%%%%%%

\subsection{Curvature}\label{par:curvature}

Let $M$ be a compact homogeneous Lorentzian manifold and $G:=\textnormal{Isom}^0(M)$.

In the case that its Lie algebra $\mathfrak{g}$ contains a direct factor isomorphic to $\mathfrak{sl}_2(\mathds{R})$, $M$ is covered by the metric product $\widetilde{M}:=N \times \widetilde{\textnormal{SL}_2(\mathds{R})}$ due to Theorem~\ref{th:homogeneous_characterization}. Here, $N$ is a compact homogeneous Riemannian manifold and $S:=\widetilde{\textnormal{SL}_2(\mathds{R})}$ is provided with the metric defined by a positive multiple of the Killing form of $\mathfrak{sl}_2(\mathds{R})$. The local geometries of $M$ and $\widetilde{M}$ coincide and it is standard to describe the curvatures of $\widetilde{M}$ through the curvatures of $N$ and $S$. One can find the results in \cite[Sections~5.3.2.2 and~5.3.2.3]{G11}. In particular, $M$ cannot be Ricci-flat in this case since $S$ is not. 

Therefore, if $G$ is not compact, it remains the case that $\mathfrak{g}$ contains a direct factor $\mathfrak{s}$ isomorphic to $\mathfrak{he}_d^\lambda$, $\lambda \in \mathds{Z}_+^d$. In the following, we will use the reductive representation given in Paragraph~\ref{par:reductive_Heisenberg} and use the same notations.

For $X \in \mathfrak{m}=\mathfrak{s}\oplus \mathfrak{p}$, we will write $X_{\mathfrak{s}}$ for the $\mathfrak{s}$-component and $X_{\mathfrak{p}}$ for the $\mathfrak{p}$-component. In a similar way, $X_{\mathfrak{m}}$ and $X_{\mathfrak{m}^\prime}$ denote the $\mathfrak{m}$- or $\mathfrak{m}^\prime$-component for $X \in \mathfrak{g}=\mathfrak{m}\oplus \mathfrak{h}$ or $X \in \mathfrak{c}=\mathfrak{m}^\prime \oplus \mathfrak{h}$, respectively.

For $X \in \mathfrak{c}=\mathds{R}Z \oplus \mathfrak{p} \oplus \mathfrak{h}$, $X_Z$ denotes the $Z$-component and for an element $X \in \mathfrak{g}=(\mathds{R}T \inplus \mathfrak{he}_d) \oplus \mathfrak{k}\oplus\mathfrak{a}$, we will write $X_T$ for the $T$-component.

Let $\left\{T, Z, X_1, Y_1, \ldots, X_d, Y_d\right\}$ be a canonical basis of $\mathfrak{s}$. By Proposition~\ref{prop:lorentz_heisenberg}~(ii), all ad-invariant Lorentzian scalar products on $\mathfrak{s}$ are equivalent. Since the scalar product $\langle \cdot, \cdot \rangle$ restricted to $\mathfrak{s} \times \mathfrak{s}$ is ad-invariant, we may assume without loss of generality that $\left\{X_1, Y_1, \ldots, X_d, Y_d\right\}$ is $\langle \cdot, \cdot \rangle$-orthonormal, $\langle T,T \rangle=0$ and $\langle T,Z \rangle=1$. Note that also $\langle Z,Z \rangle=0$.

\begin{definition}
The map $V: \mathfrak{m} \times \mathfrak{m} \to \mathfrak{m}$ is determined by \[2 \langle V(X,Y) , W\rangle=\langle [W_{\mathfrak{p}},X_{\mathfrak{p}}]_Z, Y_T \rangle+\langle X_T,[W_{\mathfrak{p}},Y_{\mathfrak{p}}]_Z \rangle\] for all $W \in \mathfrak{m}$ and for any $X,Y \in \mathfrak{m}$.
\end{definition}
\begin{remark}
$V$ is symmetric and bilinear. For all $X,Y \in \mathfrak{m}$, $V(X,Y) \in \mathfrak{p}$.
\end{remark}

\begin{definition}
In the situation above, we call the homogeneous space $M$ \textit{special} if $V \equiv 0$.
\end{definition}

\begin{proposition}\label{prop:special}
In the situation above, the following are true:
\begin{enumerate}
\item $M$ is special if and only if $[\mathfrak{p},\mathfrak{p}]_Z=\left\{0\right\}$.
\item If $M$ is special, $N$ is covered isometrically by the metric product $Z(S) \times N^\prime$ and $M$ is covered isometrically by the metric product $S \times N^\prime$. Here, $Z(S)$ is furnished with a homogeneous metric (that is unique up to multiplication with a constant), $S$ is provided with the bi-invariant metric defined in the proof of Proposition~\ref{prop:lorentz_heisenberg}~(i), and $N^\prime$ is a homogeneous Riemannian space.
\end{enumerate}
\end{proposition}
\begin{proof}
(i) If $[\mathfrak{p},\mathfrak{p}]_Z=\left\{0\right\}$, the statement is obvious.

Conversely, if $V \equiv 0$, \[\langle [W,X]_Z, T \rangle=0\] for all $W,X \in \mathfrak{p}$. Because of $\langle T,Z \rangle =1$, $[\mathfrak{p},\mathfrak{p}]_Z=\left\{0\right\}$ follows.

(ii) One can find a proof in \cite[Lemma~5.15]{G11}. Essentially, one repeats the proof of Theorem~\ref{th:geometric_characterization}. Then, the orthogonal distribution $\mathcal{O}$ orthogonal to the orbits of $S$ will be involutive due to $[\mathfrak{p},\mathfrak{p}]_Z=\left\{0\right\}$.
\end{proof}

A detailed investigation of the curvatures of $M$ is given in \cite[Section~5.3.3.4]{G11}. The formulas drastically simplify if $M$ is special. One can also describe the holonomy algebra of $M$ in this case. As an example of the results of \cite{G11}, we present the formula for the Ricci curvature of $M$ (see \cite[Proposition~5.20]{G11}).

\begin{proposition}\label{prop:Ric}
Let $\left\{W_1,\ldots,W_m\right\}$ be any orthonormal basis of $\mathfrak{p}$ and $X\in \mathfrak{m}$. Then,
\begin{align*}
\textnormal{Ric}(X,X)&=\textnormal{Ric}^S(X_{\mathfrak{s}},X_{\mathfrak{s}})+\textnormal{Ric}^N(X_{\mathfrak{p}},X_{\mathfrak{p}})-\frac{1}{(Z,Z)} (U^N(X_{\mathfrak{p}},Z),U^N(X_{\mathfrak{p}},Z))\\
&-\frac{1}{2}\sum\limits_{j=1}^m\langle X_T, [W_j,[W_j,X_{\mathfrak{p}}]_{\mathfrak{m}^\prime}]_Z \rangle
+\frac{3}{4}\sum\limits_{j=1}^m([X_{\mathfrak{p}},W_j]_Z,[X_{\mathfrak{p}},W_j]_Z)\\
&+2 \sum\limits_{j=1}^m \langle X_T,[U^N(X_{\mathfrak{p}},W_j),W_j]_Z \rangle
+\frac{1}{4}\sum\limits_{j,k=1}^m (\langle X_T, [W_k,W_j]_Z \rangle)^2\\
&-\sum\limits_{j=1}^m \langle X_T,[U^N(W_j,W_j),X_{\mathfrak{p}}]_Z\rangle
\end{align*}
Here, $\textnormal{Ric}^S$ and $\textnormal{Ric}^N$ correspond to the homogeneous spaces $S$ and $N$. The map $U^N:\mathfrak{m}' \times\mathfrak{m}' \to \mathfrak{m}'$ is defined by $2\langle U^N(A,B), C \rangle=\langle [C,A]_{\mathfrak{m}'}, B \rangle+\langle A,[C,B]_{\mathfrak{m}'} \rangle$.
\end{proposition}

From Proposition~\ref{prop:Ric} and an investigation of the Ricci curvature of the twisted Heisenberg group $S$, which can be found in \cite[Section~5.3.3.3]{G11}, it follows that for $X,Y \in \mathfrak{he}_d$, \[\textnormal{Ric}(T+X,T+Y)=\frac{1}{2} \sum\limits_{k=1}^d \lambda_k^2+\frac{1}{4}\sum\limits_{j,k=1}^m (\langle T, [W_k,W_j]_Z \rangle)^2>0,\] where $\left\{W_1,\ldots,W_m\right\}$ is any orthonormal basis of $\mathfrak{p}$ and $\lambda=(\lambda_1,\ldots,\lambda_d)$. Thus, $M$ is not Ricci-flat.

The proof of Theorem~\ref{th:homogeneous_not_Ricci_flat} is now immediate: Due to the results above, a compact homogeneous Lorentzian manifold $M$ is not Ricci-flat if the Lie algebra of its isometry group contains a direct summand isomorphic to $\mathfrak{sl}_2(\mathds{R})$. The same is true in the case of a direct summand isomorphic to $\mathfrak{he}_d^\lambda$, $\lambda \in \mathds{Z}_+^d$. So by Theorem~\ref{th:homogeneous_characterization}, $\textnormal{Isom}^0(M)$ is compact.

%%%%%%%%%%%%%%%%%%%%%%%%%%%%%%%%%%%%%%%%%%%%%%%%%%%%%%%%%%%%%%%%%%%%%%%%%%%%%%%%%%%%%%%%

\section*{Acknowledgment}

I am very grateful to Helga Baum for her dedicated guidance, her useful comments, and her strong support. I especially thank Abdelghani Zeghib for his very helpful answers to my questions.

%%%%%%%%%%%%%%%%%%%%%%%%%%%%%%%%%%%%%%%%

\bibliographystyle{unsrt}
\bibliography{Geometry_compact_homogeneous_Lorentz}

\end{document}